\documentclass[12pt]{amsart}

\usepackage{amssymb, tikz}

\usepackage{amsmath}

\newtheorem{theorem}{Theorem}[section]
\newtheorem{lemma}[theorem]{Lemma}

\newtheorem{proposition}[theorem]{Proposition}
\newtheorem{corollary}[theorem]{Corollary}
\newtheorem{claim}[theorem]{Claim}
\theoremstyle{definition}
\newtheorem{definition}[theorem]{Definition}
\newtheorem{notation}[theorem]{Notation}

\let \restr = \upharpoonright
\let \bs = \backslash
\let \into = \longrightarrow

\let \tld = \tilde

\let \sub = \subseteq
\let \elsub = \preccurlyeq

\let \av = \arrowvert
\let \ov = \overline
\let \a = \alpha
\let \b = \beta
\let \g = \gamma
\let \d = \delta
\let \e = \epsilon

\let \l = \lambda
\let \k = \kappa
\let \m = \mu
\let \n = \nu

\let \t = \theta

\let \D = \Delta

\let \s = \sigma
\let \x = \xi

\let \o = \omega

\let \P = \Pi
\let \S = \Sigma

\let \al = \aleph
\let \la = \langle
\let \ra = \rangle
\let \mtcl = \mathcal
\let \mtbb = \mathbb
\let \it = \item

\title{Forcing consequences of $\textsc{PFA}$ together with the continuum large}

\author[D. Asper\'o]{David Asper\'o}

\thanks{Mota was supported by the Austrian Science Fund FWF Project P22430. Both authors were also partially supported by Ministerio de
Educaci\'{o}n y Ciencia Project MTM2008--03389 (Spain) and by Generalitat de Catalunya Project 2009SGR--00187 (Cata\-lonia).}

\address{David Asper\'o, School of Mathematics, University of East Anglia, Norwich NR4 7TJ, UK}

\email{d.aspero@uea.ac.uk}

\author[M.A. Mota]{Miguel Angel Mota}

\address{Miguel Angel Mota,  Department of Mathematics,
University of Toronto,
Toronto, Ontario,
CANADA
M5S 2E4}

\email{motagaytan@gmail.com}

\date{}

\begin{document}


\maketitle
\pagestyle{myheadings}\markright{Forcing consequences of $\textsc{PFA}$ together with the continuum large}

\begin{abstract}
We develop a new method for building forcing iterations with symmetric systems of structures as side conditions. Using this me\-thod we prove that the forcing axiom for the class of all \emph{finitely proper} posets of size $\al_1$ is compatible with $2^{\aleph_{0}}> \aleph_2$. In particular, this answers a question of Moore by showing that $\mho$ does not follow from this arithmetical assumption.
\end{abstract}

\section{Introduction}

In the early days of forcing, Solovay and Tennenbaum (see \cite{Solovay-Tennenbaum}) developed the theory of c.c.c.\ iterations in order to show the consistency of Suslin's Hypothesis (i.e., the axiom saying that there are no Suslin trees). In fact, as they realised, their technique could be used to build models of $\textsc{ZFC}$ with the continuum arbitrarily large and
 satisfying a condition much stronger -- in the presence of $\lnot\textsc{CH}$ -- than Suslin's Hypothesis that came to be known as Martin's Axiom. Recall that a partial order has the countable chain condition (c.c.c.)\ if it has no uncountable antichains. Given a cardinal $\k$, Martin's Axiom for $\k$--many dense sets, $\textsc{MA}_\k$,  is the forcing axiom for the class of c.c.c.\ forcing notions and for collections of $\k$--many dense sets, i.e.\ the axiom saying that for any c.c.c.\ partial order $\mtbb P$ and any collection $\mtcl D$ of $\k$--many dense subsets of $\mtbb P$ there is a filter $G\sub\mtbb P$ having nonempty intersection with all members of $\mtcl D$, and Martin's Axiom is $\textsc{MA}_\k$ for all $\k<2^{\al_0}$. Martin's Axiom (typically in the form $\textsc{MA}_{\o_1}$) proved to be very successful in applications in infinite combinatorics, topology, algebra, and  other areas of mathematics (see \cite{Fremlin}).  The main features of c.c.c.\ forcing are that (1) c.c.c.\ forcing notions preserve all cardinals, and (2)  finite support iterations of c.c.c.\ forcing notions are themselves c.c.c. It follows from these two facts together that no forcing axiom of the form $\textsc{MA}_\k$ puts any upper bound on the size of the continuum (on the other hand, $\textsc{MA}_\k$ certainly implies $2^{\al_0}>\k$). 

About a decade later, the theory of proper forcing was developed by Shelah (\cite{SH100}, see also \cite{SHELAHPF}, \cite{SHELAHPIF}). A poset $\mtbb P$ is proper if for every regular cardinal $\l > \av TC(\mtbb P)\av$, every countable $N\elsub H(\l)$ such that $\mtbb P\in N$, and every $p\in\mtbb P\cap N$ there is a condition $q$ in $\mtbb P$ stronger than $p$ which is $(N,\,\mtbb P)$--generic, i.e., such that $q$ forces $\dot G\cap D\cap N\neq \emptyset$ for every dense subset $D$ (equivalently, maximal antichain) of $\mtbb P$ such that $D\in N$, where $\dot G$ is the canonical name for the generic filter. The class of proper forcings is much larger than the class of c.c.c.\ forcings. Nevertheless, proper forcings are well behaved in the sense that (1) they preserve $\o_1$, and that (2) every countable support iteration with proper iterands is itself proper.  Baumgartner showed the consistency of the forcing axiom for the class of proper forcings and for collections of $\al_1$--many dense sets, also known as the Proper Forcing Axiom ($\textsc{PFA}$), by building a countable support iterations of proper forcing notions of length a supercompact cardinal (see \cite{BAU}). Soon it was realised that $\textsc{PFA}$ has many consequences (see for example \cite{BAU} for a classic overview). One remarkable consequence of $\textsc{PFA}$ (by Todor\v{c}evi\'{c} and Veli\v{c}kovi\'{c}, see \cite{Bekkali} and \cite{Velickovic}) is that, unlike any forcing axiom of the form $\textsc{MA}_\k$, it does decide the value of $2^{\al_0}$; in fact, $\textsc{PFA}$ implies $2^{\al_0} = \al_2$.\footnote{The first derivation of $2^{\al_0} = \al_2$ from a natural forcing axiom was the proof, due to Foreman, Magidor and Shelah (see \cite{FMS}) that \emph{Martin's Maximum}, which is a provably maximal forcing axiom for collections of $\al_1$--many dense sets and is strictly stronger than $\textsc{PFA}$, implies $2^{\al_0} = \al_2$.}  More recently, Moore has proved (see \cite{MOORE}) that $2^{\al_0}=\al_2$ follows already from the bounded form of $\textsc{PFA}$ known as $\textsc{BPFA}$ ($\textsc{BPFA}$ can be phrased as the axiom saying that $\la H(\o_2), \in\ra$ is a $\S_1$ elementary substructure of the structure $\la H(\o_2), \in\ra$ as computed in any generic extension by a proper forcing). 

Given that strong forcing axioms typically imply $2^{\aleph_{0}}=\al_2$, a natural problem when faced with a consequence $\s$ of a forcing axiom is to find out whether $\s$ itself has any impact on the size of the continuum and which. 
The standard strategy for producing models of $\P_2$ consequences $\s$ (over the structure $\la H(\o_2), \in, \o_1 \ra$)  of forcing axioms is by means of forcing
iterations in which one keeps adding witnesses of the relevant $\S_1$ facts.\footnote{This is also the traditional way of building models of actual forcing axioms like $\textsc{MA}_\k$, $\textsc{BPFA}$, $\textsc{PFA}$, or Martin's Maximum.} If it can be shown that there is always a forcing adding these witnesses which moreover has the c.c.c., then a sufficiently long finite support iteration of (carefully chosen) instances of this forcing will produce a model of $\s$. Since finite support iterations of c.c.c.\ forcings are themselves c.c.c.\ and since c.c.c.\ forcings preserve cardinals, such a construction will give rise to models of $\s$ in which $2^{\al_0}$ can attain (almost) any arbitrarily fixed value. For example, this is the standard way of showing that Suslin's Hypothesis is consistent with $2^{\al_0}$ being any arbitrarily fixed $\al_\a$ with $cf(\al_\a)>\o_1$ (see \cite{Solovay-Tennenbaum}).\footnote{As proved by Jensen  (\cite{DJ}), Suslin's hypothesis is also consistent with $\textsc{CH}$, but the proof of this uses a countable -- rather than finite --  support iteration. Shelah's  \cite{SHELAHPIF} is a classical reference on forcing $\P_2$ statements over $\la H(\o_2),  \in, \o_1\ra$ together with $\textsc{CH}$ by means of countable support iterations. Also, Laver proved in \cite{Laver} that adding any number of random reals to any model of $\textsc{MA}_{\al_1}$ preserves Suslin's Hypothesis, and therefore Suslin's Hypothesis is consistent with $2^{\al_0}$ being singular of cofinality $\o_1$.} In fact, if $\textsc{GCH}$ holds and $\k$ is any cardinal of uncountable cofinality, then a certain finite support iteration of length $\k$ will produce a model of $2^{\al_0}=\k$ together with $\textsc{MA}_\l$ for all $\l<cf(\k)$, and  already $\textsc{MA}_{\o_1}$ implies Suslin's Hypothesis (\cite{Solovay-Tennenbaum}).  

However, one often deals with statements $\s$ that cannot be changed by c.c.c.\ forcing. Consider for example \emph{Club Guessing} ($\textsc{CG}$), which says that there is a ladder system on $\o_1$ (a \emph{ladder system} is a sequence $\langle A_{\delta} \,:\, \delta \in Lim(\o_1)\rangle$ such that each $A_\delta$ is a cofinal subset of $\delta$ of order type $\o$) that `guesses' clubs, in the sense that if $C\sub\o_1$ is a club, then there is some limit ordinal $\d<\o_1$ such that a final segment of $A_\d$ is contained in $C$. Club Guessing is clearly a consequence of Jensen's $\diamondsuit$. However, unlike $\diamondsuit$, Club Guessing is immune to c.c.c.\ forcing. In fact, every club of $\o_1$ in any generic extension by a c.c.c.\ forcing contains a club from the ground model, and therefore Club Guessing holds in any extension by any c.c.c.\ forcing if it happens to hold in the ground model. Now take our statement $\s$ to be the negation of Club Guessing. It follows from the above observation that the strategy of producing a model of $\s$ by building a finite support iteration in which the iterands  have the c.c.c.\ cannot work.  However, $\lnot\textsc{CG}$ is consistent. It follows from $\textsc{BPFA}$, and it can be forced over any model of $\textsc{GCH}$ by  a countable support iteration of proper forcings of length $\o_2$. In fact, given a ladder system $\mtcl A = \langle A_{\delta} \,:\, \delta \in Lim(\o_1)\rangle$, the natural forcing for adding, by initial segments, a club $C\sub\o_1$ `avoiding' $\mtcl A$ -- in the sense that $C\cap A_\d$ is finite for all $\d\in Lim(\o_1)$ -- is proper.\footnote{It follows from this that in fact the negation of the weaker principle known as Weak Club Guessing (see below) follows from $\textsc{BPFA}$ and can be forced over any $\textsc{GCH}$ model.}  On the other hand, although countable support iterations of proper forcings are always proper, finite support iterations of infinite length of forcings that are not c.c.c.\ always collapse $\o_1$. It follows that any (standard) forcing construction for producing a model of $\lnot\textsc{CG}$ by iterating instances of the proper forcing for adding clubs avoiding given ladder systems will have to be an iteration with countable supports rather than finite supports,  and therefore will never give rise to a model with $2^{\aleph_{0}}>\al_2$. 
The
reason is the well--known general fact that for any countable support iteration $\la\mtbb P_\xi\,:\,\x\leq\l\ra$ of non-trivial forcings and any ordinal $\xi$, if $\l\geq\x+\o_1$, then $\mtbb P_\l$ forces over $V$ that there is a surjection from $\o_1^V$ onto the reals of $V[\dot G_\x]$ (where of course $\dot G_\x$ denotes the canonical $\mtbb P_\x$--name for the generic filter). In particular, if $cf(\l)\geq\o_1$ and $\mtbb P_\l$ has the $\l$--c.c., then $\mtbb P_\l$ forces $2^{\al_0}\leq\al_2$. To sum up, c.c.c.\ forcing is useless when it comes to forcing the negation of club--guessing principles over models satisfying these club--guessing principles and, on the other hand, countable support iterations of proper forcing notions can easily give rise to generic extensions satisfying the negation of club--guessing principles, but $2^{\al_0}  \leq\al_2$ must necessarily hold in those extensions. 

In view of these considerations it is natural to enquire whether various failures of Club Guessing on $\o_1$ are consistent with the continuum large.
In some cases, this question can be settled by taking a model of the property in question together with the continuum small and arguing that adding many Cohen or random reals to it preserves the property.
For example, it can be proved that the very statement we have been considering above, namely $\lnot\textsc{CG}$, is indeed compatible with $2^{\al_0}>\al_2$. In fact it is not difficult to prove (and possibly folklore) that the product with finite supports of Cohen forcing always preserves $\lnot \textsc{CG}$.

There are other anti--diamond principles for which the strategy of adding many Cohen reals does not work. This is for example the case for the negation of \textit{Weak Club Guessing}. Weak Club Guessing ($\textsc{WCG}$) says that there exists a ladder system $\mtcl A=\langle A_{\delta} : \delta \in Lim(\o_1) \rangle$ with the property that for every club $C \subseteq \omega_{1}$ there is some $\delta \in C$ such that $A_{\delta} \cap C$ is infinite. 
Note that Club Guessing implies Weak Club Guessing and that, by what we have already mentioned, $\lnot \textsc{WCG}$ is a consequence of $\textsc{BPFA}$ and can be forced over any model of $\textsc{GCH}$ by a countable support iteration of proper forcings of length $\o_2$. On the other hand, Cohen forcing always adds a ladder system witnessing $\textsc{WCG}$. This was originally proved by Juhasz in \cite{JUHASZ}, where he showed that a weakening of $\clubsuit$ implying Weak Club Guessing always holds after adding  a Cohen real.  Hence, the consistency of $\lnot \textsc{WCG}$
with $2^{\al_0}>\al_2$ cannot be proved by adding Cohen reals to a model where Weak Club Guessing is
false. However, one possibility for this is to add many random reals to a model of $\lnot \textsc{WCG}$. In fact, it is not hard to see (and, again, possibly folklore)
that random forcing always preserves $\lnot \textsc{WCG}$.

There are however strengthenings of $\lnot \textsc{CG}$ for which the above methods do not work. Consider for example the conjunction of $\lnot\textsc{WCG}$ and $\textsc{MA}_{\o_1}$. This theory implies several strong forms of $\lnot\textsc{WCG}$.\footnote{One such strong form of $\lnot\textsc{WCG}$ is for instance $Code(\textrm{even--odd})$, a principle formulated by Miyamoto saying that for every ladder system $\mtcl A=\langle A_{\delta} : \delta \in Lim(\o_1)\rangle$ and every
$B\subseteq \omega_{1}$ there are two clubs $C$ and $D$ of $\omega_{1}$ such that for each $\delta \in C$, if
$\delta \in B$ (resp.\ $\delta \notin B$), then $\av A_{\delta} \cap D\av <\aleph_{0} $ is odd (resp.\ even).}
 As we said, $\lnot\textsc{WCG}$ is preserved after adding random reals, but $\textsc{MA}_{\o_1}$ will fail in the resulting model.\footnote{One reason why $\textsc{MA}_{\o_1}$ fails after adding random reals is that in the extension there is a c.c.c.\ partial order whose product with itself is not c.c.c. This result is due to Kunen and a proof can be found in \cite{Roitman}.}

Another example of a strengthening of $\lnot \textsc{CG}$ that cannot be forced easily with a large continuum is the negation of $\mho$ (mho). The principle $\mho$, formulated by Moore, says that there is a sequence $\langle f_{\alpha} : \alpha \in \omega_{1}\rangle$ such that $f_{\alpha}$ is a continuous map, with respect to the order topology, from $\alpha$ into $\omega$ for all $\alpha \in \omega_{1}$, and with the property that for every club $E\subseteq \omega_{1}$ there is a $\delta$ in $E$ such that $f_{\delta}$ takes all values in $\omega$ on $E \cap \delta$. The following is an observations of Moore
concerning this statement: Notice that if $\alpha < \omega_{1}$ and $f:\alpha \rightarrow \omega$ is continuous, then $\alpha$ can be partioned into clopen intervals on which $f$ is constant. In such a situation there is a cofinal $C \subseteq \alpha$ of order-type at most $\omega$ such that $f(\varepsilon)$ depends only on the size of $\varepsilon \cap C$.
From this it is clear that $\mho$ follows from $\textsc{CG}$. In \cite{MOORE2} Moore shows that $\mho$ implies the existence of an Aronszajn line containing no Countryman type, and asks whether $\mho$ follows from $2^{\aleph_{0}}> \aleph_{2}$. One motivation for this question is that, by the above implication, $\mho$ entails that there is no basis for the uncountable linear orders containing exactly $5$ uncountable members.\footnote{If there is such a basis for the uncountable linear orders, then there must be a Countryman type $C$ such that every Aronszajn line contains a copy of $C$ or of the reverse of $C$ (or of both).} On the other hand, Moore proved in \cite{MOORE5} that the existence of such a basis is consistent with $\textsc{ZFC}$ and that in fact it follows from $\textsc{PFA}$ that there is such a basis. Hence, if $\mho$ could be derived from $2^{\al_0}>\al_2$, then the existence of a $5$ element basis for the uncountable linear orders would imply $2^{\al_0}=\al_2$ (it is easy to see that it implies $2^{\al_0}>\al_1$). 

It should be noted that $\textsc{MA}_{\o_1}$, $\lnot \textsc{WCG}$ and $\lnot\mho$ follow from the forcing axiom for the class of all proper posets of size $\al_1$ (which we will call $\textsc{PFA}(\o_1)$).
In this paper we introduce an alternative method to standard countable support iterations for producing models of certain $\P_2$ statements. Using this method we prove that a certain forcing axiom which is a natural fragment of $\textsc{PFA}(\o_1)$ and which implies the three statements above -- $\textsc{MA}_{\o_1}$, $\lnot \textsc{WCG}$, and $\lnot\mho$ -- is consistent together with  $2^{\al_0}>\al_{2}$.\footnote{Concerning the connection mentioned before between $\mho$ and the (non)existence of a $5$ element basis for the uncountable linear orders, we should point out that it is still open whether the existence of such a basis is compatible with $2^{\al_0}>\al_2$.} In fact, we build a cardinal--preserving generic extension where this fragment of $\textsc{PFA}(\o_1)$ holds and $2^{\aleph_{0}}$ is equal to $\kappa$, where $\kappa$ is an arbitrarily fixed cardinal satisfying certain $\textsc{GCH}$ like assumptions in the ground model.

\begin{definition}
Given a poset $\mtbb P$, we will say that $\mtbb P$ \emph{is finitely proper} if and only if for every regular cardinal $\lambda>\av TC(\mtbb P)\av$, every finite set $\{N_i\,:\, i \in m\}$ of countable elementary substructures of $H(\lambda)$ containing $\mtbb P$ and every condition $p\in \bigcap\{N_i\, :\, i < m\} \cap\mtbb P$ there is a $\mtbb P$--condition extending $p$ and $(N_i,\, \mtbb P)$--generic for all $i$.
\end{definition}

\begin{definition}
Let $\textsc{PFA}^{\mbox{fin}}(\o_1)$ denote the forcing axiom for the class $\Gamma$ of all finitely proper posets of size $\al_1$ and for families of $\al_1$--many dense sets. More precisely, $\textsc{PFA}^{\mbox{fin}}(\o_1)$ says that whenever $\mtbb P \in \Gamma$ and $\{D_{\alpha} \,:\, \alpha \in \aleph_{1}\}$ is a set of dense subsets of
$\mtbb P$, there is a filter $G$ on $\mtbb P$ meeting every $D_{\alpha}$.
\end{definition}

Note that c.c.c.\ partial orders are finitely proper. Indeed, if $\mtbb P$ is c.c.c.\ and $N$ is any countable elementary substructure of $H(\l)$, for any cardinal $\l>\av TC(\mtbb P)\av$,  such that $\mtbb P\in N$, then any condition in $\mtbb P$ is $(N,\,\mtbb P)$--generic, simply because $A\sub N$ for any maximal antichain of $\mtbb P$ in $N$. Therefore the forcing axiom $\textsc{PFA}^{\mbox{fin}}(\o_1)$ is a generalization of $\textsc{MA}_{\o_1}$. Also, unlike any form of Martin's Axiom, $\textsc{PFA}^{\mbox{fin}}(\o_1)$ does have a strong impact on the club filter on $\o_1$. Specifically, this forcing axiom implies the failure of both $\textsc{WCG}$ and $\mho$. The proofs of these implications can be found in Section \ref{applications}.

Our main theorem is the following.

\begin{theorem}\label{mainthm} ($\textsc{CH}$) If $\k$ is a cardinal such that $\k^{\al_1}= \k$ and $2^{<\k}=\k$, then there exists a proper forcing notion $\mtcl P$ with the $\al_2$--chain condition such that both $\textsc{PFA}^{\mbox{fin}}(\o_1)$ and $2^{\al_0}=\k$ hold in the generic extension by $\mtcl P$.
\end{theorem}

Our method produces a proper forcing notion with the $\al_2$--chain condition. This forcing notion $\mtcl P$ is the direct limit $\mtcl
P_{\k}$ of a sequence $\langle \mtcl P_{\alpha}: \alpha < \k \rangle$ of partial orders, where $\mtcl P_{\alpha}$ is a complete
suborder of $\mtcl P_{\beta}$ whenever $\alpha$ is less than $\beta$. Our construction can thus be seen as a forcing iteration in a broad sense.

One crucial feature in the proof of properness is the use of certain finite ``symmetric systems'' of countable structures as side conditions. These structures will be elementary substructures of $H(\kappa)$ and will be added by $\mtcl P_{0}$. If $N$ is one of them, $q=(F, \D) \in \mtcl P_\a$,  $(N, \alpha) \in \D$, and $N$ belongs to a club of ``sufficiently correct'' structures, then all relevant pieces of information coming from any $\mtcl P_\a$--extension of $q$ can be relativized to $N$. This is the case essentially because, under the above assumptions, if $\x\in dom(F)$ -- i.e., $F(\x)$ carries nontrivial information on the (finitely proper) poset $\Phi(\x)$ with domain included in $\o_1$ picked by our bookkeeping --  and $\x\in \a \cap N$, then $F(\x)$ is asked to be generic over $N[\dot G_\x]$ with respect to  $\Phi(\x)$. The domain of $F$ will be finite, and for every $\xi\in dom(F)$,  $F(\x)$ will be a $\mtcl P_\x$ --name for a countable ordinal. The general technique of ensuring properness of a given forcing notion by explicitly incorporating elementary substructures of some large enough model as side conditions may be traced back to Todor\v{c}evi\'{c}'s
\cite{TO}. The more specific approach of considering symmetric systems of countable structures as side conditions in contexts  in which one starts with a model of $\textsc{CH}$ and wants to obtain a forcing notion which is proper and has the $\al_2$--chain condition is quite natural. In fact, this approach has already shown up in several places in the literature prior to our work (see for example \cite{Abraham-Cummings}, \cite{KOSZMIDER} and \cite{TO1}). The main novelty of our present work is that it incorporates the use of symmetric systems of structures as side conditions affecting all iterands of a given forcing iteration (or of an initial segments thereof) rather than a single forcing as in the above references. It is worth pointing out that Neeman (\cite{Neeman}) has developed a different method for building proper forcing notions by means of finite support iterations with side conditions. His side conditions are $\in$--chains of certain types of objects rather than symmetric systems of countable structures. The members of Neeman's side conditions may be countable elementary substructures but may be also of a different nature. One important difference between his work and ours is that in ours we strive to obtain a forcing notion with the $\al_2$--c.c., which is the reason why we cannot do with $\in$--chains of structures and need symmetric systems instead, whereas the $\al_2$--c.c.\ typically does not hold in Neeman's constructions.\footnote{He typically does need some $\k$--c.c., for larger $\k$, which he tends to achieve thanks to the use of structures of the form $H(\a)$ in his side conditions.}

We feel that the main contribution of the present paper is not so much a particular consistency result as the introduction of a fairly general method for building interesting forcing constructions. In fact, we have found further applications of (variations of) our method since this paper was first written and circulated in 2010. 
For instance, in \cite{genMA} we build a model of a generalisation of Martin's Axiom to a certain natural class of forcing notions with the $\al_2$--chain condition, with no restriction on their size. This generalisation of Martin's Axiom implies certain interesting `uniform' failures of Club Guessing whose consistency we don't know how to prove by methods other than ours.

There are natural weakenings of $\mho$ whose negation does not seem to hold in the models built by the methods in the present paper. Specifically, given $n<\o$ let $\mho_n$ be the principle saying that there is a sequence $(f_\d)_{\d\in\o_1}$ with $f_\d:\d\into n$ continuous function for each $\d$ such that for every club $C\sub\o_1$ there is some $\d\in C$ such that $f_\d^{-1}(j)\cap C$ is unbounded in $\d$ for all $j<n$. Clearly, for all $2\leq n<m<\o$, $\mho$ implies $\mho_m$, and  $\mho_m$ implies $\mho_n$ (these weakenings of $\mho$ have been defined also by Moore). Our present methods do not seem to produce models of $\lnot\mho_n$ for any $n$ (see Section \ref{applications} for a brief discussion of this).  
We should point out that, even if none of the principles $\lnot\mho_n$ is known to follow from $\textsc{PFA}^{\mbox{fin}}(\o_1)$, already $\lnot\mho_2$ certainly follows from $\textsc{PFA}(\o_1)$ (see the remark in Section \ref{applications}). 

The rest of the paper is organized as follows: In Section \ref{symmetric_systems} we introduce the notion of symmetric system of structures and prove basic properties of this notion that we will use throughout the paper. In Section \ref{a_general_construction} we present a rather general construction of a finite support forcing iteration using symmetric systems of structures as side conditions and prove several facts applying to this general context. Section \ref{the_forcing_construction} starts with the definition of a partial order $\mtcl P$ that will be shown to witness the conclusion of Theorem \ref{mainthm}. This partial order is a special case of the construction in Section \ref{a_general_construction}. We then prove the relevant facts of $\mtcl P$ not covered by the general theory in Section \ref{a_general_construction}. It follows from these facts that $\mtcl P$ indeed witnesses the conclusion of Theorem \ref{mainthm}. Finally, in Section \ref{applications} we show that $\textsc{PFA}^{\mbox{fin}}(\o_1)$ implies the failure of Weak Club Guessing and of $\mho$.

Even if this work tries to be reasonably self--contained, we will assume that the reader has a good knowledge of forcing, and in particular some familiarity with proper forcing. Two good references are Kunen (\cite{KUNEN}) and Jech (\cite{JECH}). Most of our notation is standard, and we have tried to give complete explanations of the relevant symbols and notions whenever we deviate from the standard use.

\textbf{Acknowledgments}:  We wish to thank Hiroshi Sakai for showing us a proof that the negation of Club Guessing is preserved by any product with finite supports of Cohen forcing and a proof that Cohen forcing adds a Weak Club Guessing sequence, and Michael Hrusak for showing us a proof that adding random reals preserves $\lnot\textsc{WCG}$.
 We thank two anonymous referees for urging us to isolate a toolbox of basic lemmas and write their proofs with full resolution. This has indeed made the paper more easily readable. Finally, we thank the referees for observing that $Code(\textrm{even--odd})$ follows from $\textsc{MA}_{\o_1}$ together with $\lnot\textsc{WCG}$.\footnote{In an earlier version of the paper we were focusing on $Code(\textrm{even--odd})$ rather than the stronger (and more natural) $\textsc{MA}_{\o_1} + \lnot\textsc{WCG}$.} 

\section{Symmetric systems}\label{symmetric_systems}

Our forcing $\mtcl P$ for proving Theorem \ref{mainthm}, to be defined in Section \ref{the_forcing_construction}, will be the direct limit $\mtcl P_{\k}$ of a certain sequence $\la\mtcl P_\a\,:\, \a < \k\ra$ of forcings. The properness of each $\mtcl P_\a$ will be witnessed by a certain club $\mtcl M^\ast_\a$ of $[H(\t_\a)]^{\al_0}$ for some high enough cardinal $\t_\a$ (see Section \ref{the_forcing_construction}). The main idea here is to use the elements of $\mtcl M_\a$ -- where $\mtcl M_\a$ is the club of restrictions to $H(\k)$ of members of $\mtcl M^\ast_\a$ -- as side conditions to ensure properness, but without losing the $\al_2$--chain condition. This brings us to the notion of \emph{symmetric system} of structures. As we mentioned in the introduction, the notion of symmetric system of structures is a natural one in the context of building forcing notions, over models of $\textsc{CH}$, which are intended to be both proper and with the $\al_2$--chain condition. In this section we define this notion and analyse its basic properties, which we will repeatedly use throughout the rest of the paper.
This type of analysis can be found also for example in \cite{Abraham-Cummings}, \cite{KOSZMIDER} and \cite{TO1}, where the notion of symmetric system shows up too (with different names).

Here, and in the remainder of the paper, we adopt the convention of denoting by $\d_N$ the ordinal $N\cap\o_1$ if $N$ is a set such that $N\cap\o_1$ is an ordinal. Also, in this section $\k$ can be taken to be the same $\k$ that has been fixed in the statement of Theorem \ref{mainthm}, but everything here works the same with any other choice of $\k$ (as long as $\k\geq\o_2$ and $\k$ is a cardinal).\footnote{The theory works also for the case $\k=\o_1$ but this is a degenerate case in which symmetric systems are simply finite $\in$--chains of countable transitive models.} 

\begin{definition}\label{symm}
Let $P\sub H(\k)$, and let $\{N_i\,:\,i<m\}$ be a finite set of countable subsets of $H(\k)$. We will say that \emph{$\{N_i\,:\,i<m\}$ is a $P$--symmetric system} if

\begin{itemize}

\it[$(A)$] For every $i<m$, $(N_i, \in, P)$  is an elementary substructure of $(H(\k), \in, P)$.

\it[$(B)$] Given distinct $i$, $i'$ in $m$, if $\d_{N_i}=\d_{N_{i'}}$, then there is a (unique) isomorphism $$\Psi_{N_i, N_{i'}}:(N_i, \in, P)\into (N_{i'}, \in, P)$$ 

\noindent Furthermore, we ask that $\Psi_{N_i, N_{i'}}$ be the identity on $N_i\cap N_{i'}$.

\it[$(C)$] For all $i$, $j$ in $m$, if $\d_{N_j}<\d_{N_i}$, then there is some $i'<m$ such that $\d_{N_{i'}}=\d_{N_i}$ and $N_j\in N_{i'}$.

\it[$(D)$] For all $i$, $i'$, $j$ in $m$, if $N_j\in N_i$ and $\d_{N_i}=\d_{N_{i'}}$, then there is some $j'<m$ such that $\Psi_{N_i, N_{i'}}(N_j)=N_{j'}$.

\end{itemize}

\end{definition}

In (A) in the above definition, and elsewhere, we will tend to refer to structures  $(N, \in, P\cap N)$ by the simpler expression $(N, \in, P)$.

\begin{lemma}\label{iso1}
Let $P\sub H(\kappa)$ and let $N$, $N'$ and $M$ be countable elementary substructures of $(H(\kappa), \in, P)$. Suppose $M \in N$ and $\Psi_{N, N'}: (N, \in, P)  \into (N', \in, P)$ is an isomorphism. Then $\Psi_{N, N'}(M)$  is also a countable elementary substructure of $(H(\kappa), \in, P)$.
\end{lemma}

\begin{proof}

First note that $\Psi_{N, N'} \restr M$ is an isomorphism between $(M, \in, P)$ and $(\Psi_{N, N'}(M), \in, P)$. Assume now that $\overrightarrow{y}$ is a finite vector of elements of $\Psi_{N, N'}(M)$ and that the formula $\exists x \varphi(x, \overrightarrow{y})$ is true in $(H(\kappa), \in, P)$. We show that there is some $z\in\Psi_{N, N'}(M)$ such that $(H(\kappa), \in, P)\models \varphi(z, \overrightarrow{y})$. But $\overrightarrow{y}$ is also a finite vector of elements of $N'$ and by correctness of $(N', \in, P)$, the formula $\exists x \varphi(x, \overrightarrow{y})$ is true in $(N', \in, P)$. Let $\overleftarrow{y}$ be the vector of elements of $N$  such that $\Psi_{N, N'}(\overleftarrow{y})=\overrightarrow{y}$, and note that the formula $\exists x \varphi(x, \overleftarrow{y})$ is true in $(N, \in, P)$ since the map $\Psi^{-1}_{N, N'}:(N', \in, P)\into (N, \in, P)$ is an isomorphism. Furthermore, by correctness of $(N, \in, P)$, the formula $\exists x \varphi(x, \overleftarrow{y})$ is also true in $(H(\kappa), \in, P)$. From this, and using the fact that $\overleftarrow{y}$ is also a vector of elements of $M$, we conclude that there exists an element $z'$ in $M$ such that the formula $\varphi(z', \overleftarrow{y})$ is true in $(M, \in, P)$. It suffices to let $z =\Psi_{N, N'}(z')$.
\end{proof}

\begin{lemma}\label{iso2}
Let $P\sub H(\k)$, let $\mtcl N$ be a $P$--symmetric system, and let $N\in\mtcl N$.

\begin{itemize}
\it[(i)] If $M_1 \in \mtcl N \cap N$ and there is some $M_2 \in \mtcl N$ (not necessarily in $\mtcl N \cap N$) such that $\delta_{M_1} < \delta_{M_2} < \d_N$, then there is some $M_3$ in $\mtcl N \cap N$ such that $M_1 \in M_3$ and $\d_{M_3}=\d_{M_2}$.

\it[(ii)] In particular, $\mtcl N \cap N$ is also a  $P$--symmetric system.

\it[(iii)] If $\mtcl W \subseteq  N$ is a $P$--symmetric system and $\mtcl N \cap N  \subseteq \mtcl W$, then $$\mtcl V:= \mtcl N \cup\{\Psi_{N, N'}(W) \,:\, W \in \mtcl W,\, N' \in \mtcl N,\,\d_{N'}=\d_N\}$$ is a $P$--symmetric system.
\end{itemize}
\end{lemma}

\begin{proof}
We start with (i). Since $\mtcl N$ satisfies clause $(C)$ in the definition of symmetric system, there exist some $M_4$ and $N'$ in $\mtcl N$ such that $M_1 \in M_4 \in N'$, $\delta_{M_4}= \delta_{M_2}$, and $\delta_{N'}= \delta_{N}$.  Since $(N', \in, P)$ and $(N, \in, P)$ are isomorphic and the corresponding isomorphism $\Psi_{N', N}$ fixes $N \cap N'$ (in particular, it fixes $M_1$), it follows that $M_1 \in \Psi_{N', N}(M_4)$. Finally, note that clause $(D)$ implies that $M_3:=\Psi_{N', N}(M_4)$ is in $\mtcl N \cap N$, and that $\d_{M_3}=\d_{M_4}=\d_{M_2}$ again as $\Psi_{N', N}$ fixes $N\cap N'$.

Let us move on to conclusion (ii).  $\mtcl N \cap N$ satisfies clauses $(A)$, $(B)$ and $(D)$ of Definition \ref{symm} since they hold for $\mtcl N$, and $(C)$ follows from (i).

We prove now (iii). By Lemma \ref{iso1}, $\mtcl V$ satisfies clause $(A)$ of Definition \ref{symm}. 

Let us check now that $\mtcl V$ satisfies $(B)$. So, let $V_1$ and $V_2$ in $\mtcl V$ be such that $\delta_{V_1}=\delta_{V_2}$. We must show that $(V_1, \in, P)$ and $(V_2, \in, P)$ are isomorphic and that the corresponding isomorphism fixes $V_1 \cap V_2$. Without loss of generality we can assume that $\delta_{V_1}< \delta_{N}$ (note that if $\delta_{V_1}\geq  \delta_N$, then $V_1$ and $V_2$ are elements of $\mtcl N$) and that $V_1$ or $V_2$ are not in $N$. Note also that $\mtcl V \cap N \subseteq \mtcl W$: If $W \in\mtcl W$, $N' \in\mtcl N$, $\d_{N'}=\d_N$, and $\Psi_{N, N'}(W)\in N$, then $\Psi_{N, N'}(W)\in N\cap N'$ and therefore $\Psi_{N, N'}(W) = W$ as $\Psi_{N, N'}$ is the identity on $N\cap N'$. It follows that there are $N_1$ and $N_2$, both in $\mtcl N$,  such that $\delta_{N_i} = \delta_{N}$ and $V_i \in N_i$ ($i \in\{1,2\}$). Let $V'_i=\Psi_{N_i, N}(V_i)$ ($i\in\{1, 2\}$). It follows that the map $\Psi_{N, N_2}\circ \Psi_{V'_1, V'_2}\circ (\Psi_{N_1, N}\restr V_1)$ is an isomorphism $\Psi_{V_1, V_2}$ between $(V_1, \in, P)$ and $(V_2, \in, P)$, and of course it is the unique isomorphism between these structures.  Let $\tld V_1 = \Psi_{N_1, N_2}(V_1)=\Psi_{N, N_2}(V'_1)$, and note that $\Psi_{V_1, V_2}=\Psi_{\tld V_1, V_2}\circ (\Psi_{N_1, N_2}\restr V_1)$. To see that $\Psi_{V_1, V_2}$ fixes $V_1\cap V_2$, note that if $x\in   V_1\cap V_2$, then $x=\Psi_{N_1, N_2}(x)=x$ (since $x\in N_1\cap N_2$), and therefore $\Psi_{V_1, V_2}(x)=\Psi_{\tld V_1, V_2}(x)=x$. The last equality holds since $x\in \tld V_1\cap V_2$ and since $\Psi_{\tld V_1, V_2}$ fixes $\tld V_1\cap V_2$ (this is true since $\Psi_{\tld V_1, V_2}=\Psi_{N, N_2}(\Psi_{V'_1, V'_2})$ and since $\Psi_{V'_1, V'_2}$ fixes $V'_1\cap V'_2$). 

The next step is to prove that $\mtcl V$ satisfies $(C)$. So, assume that  $V_1$, $V_2$  are elements of $\mtcl V$ such that $\delta_{V_1} < \delta_{V_2}$. We must show that there is a $V_3$ in $\mtcl V$ containing $V_{1}$ and such that $\delta_{V_2}=\delta_{V_3}$. Note that if $\delta_{V_1} \geq \delta_{N}$, then $V_1$ and $V_2$ are elements of the symmetric system $\mtcl N$ and we are done. Also note that if $\d_{V_2} \geq \delta_{N}> \delta_{V_1}$, then $V_2\in\mtcl N$ and there is some $N_3$ in $\mtcl N$ such that $V_1\in N_3$ and $\delta_{N_3}=\delta_N$. In particular there is some $V_3\in\mtcl N$ such that $\d_{V_3}=\d_{V_2}$ and $N_3\in V_3$, and therefore such that $V_1\in V_3$. So, we may assume that $\delta_{V_2}< \delta_{N}$ and that there are $W_1$, $W_2$ in $\mtcl W$, together with $N_1$, $N_2$ in $\mtcl N$, such that $V_1 = \Psi_{N, N_1}(W_1)$ and $V_2 = \Psi_{N, N_2}(W_2)$. Since $\mtcl W$ is a symmetric system, there exists some $W_3 \in \mtcl W$ such that $W_1 \in W_3$ and $\d_{W_3}=\d_{W_2}$. It suffices to let $V_3= \Psi_{N, N_1}(W_3)$. 

Finally, we check  that
$\mtcl V$ satisfies $(D)$.  Let $V_1$, $V_2$ and $V_3$ be elements of $\mtcl V$ such that $\delta_{V_1} < \delta_{V_2}= \delta_{V_3}$ and $V_1 \in V_2$. We must show that $\Psi_{V_2,V_3}(V_1) \in \mtcl V$.  Note that if $\d_{V_1} \geq \d_N$, then $V_1$, $V_2$ and $V_3$ are in $\mtcl N$ and we are done.  So, we can assume that $\delta_{V_1} < \delta_{N}$. Now note that if $\delta_{V_1} < \delta_{N} \leq \d_{V_2}$, then there are $N_2$, $N_3$ $\in \mtcl N$ and $W \in \mtcl W$  such that $V_1 = \Psi_{N, N_2}(W)$ and such that, for all $j \in \{2, 3\}$, $\delta_{N_j}=\delta_N$ and $N_j \subseteq V_j$, and $N_3=\Psi_{V_2, V_3}(N_2)$ if $\d_{N_2}<\d_{V_2}$. Hence, $\Psi_{V_2, V_3}(V_1)= \Psi_{N_2, N_3}(V_1)=\Psi_{N_2, N_3}(\Psi_{N, N_2}(W))= \Psi_{N, N_3}(W) \in \mtcl V$. The last case of this proof is when $\d_{V_2} < \d_N$.  In this case there are $W_i \in \mtcl W$ and $N_i \in \mtcl N$ ($i \in \{1, 2, 3 \}$) such that $\d_{N_i}=\d_N$ and $V_i= \Psi_{N, N_i}(W_i)$. Furthermore, since $V_1 \in N_1 \cap N_2$ and $\Psi_{N_1, N_2}$ fixes $N_1 \cap N_2$, we also have that $\Psi_{N_2, N}(V_1)= \Psi_{N_2, N}(\Psi_{N_1, N_2}(V_1)) = \Psi_{N_1, N}(V_1) = W_1$.
Since $\mtcl W$ is a symmetric system, we know that $W_4:= \Psi_{W_2, W_3}(W_1)$ is in $\mtcl W$. It thus follows that
$\Psi_{V_2, V_3}(V_1)=   \Psi_{W_3, V_3}(\Psi_{W_2, W_3}((\Psi_{V_2, W_2}(V_1)))=  \Psi_{W_3, V_3}(\Psi_{W_2, W_3}(\Psi_{N_2, N}(V_1))) = \Psi_{W_3, V_3}(\Psi_{W_2, W_3}(W_1)) = \Psi_{N, N_3}(W_4) \in \mtcl V$ (The second equality is true since $\Psi_{V_2, W_2}$ is the same thing as $\Psi_{N_2, N}\restr V_2$, the third equality is true since $\Psi_{N_2, N}(V_1) = W_1$, and the last equa\-lity is true since $W_4= \Psi_{W_2, W_3}(W_1)$ and since $\Psi_{W_3, V_3}$ is the restriction of  $\Psi_{N, N_3}$ to $W_3$.)
\end{proof}

\begin{lemma}\label{iso3}

Let $P\sub H(\k)$ and let $\mtcl N_0=\{N^0_i\,:\,i<m\}$ and $\mtcl N_1=\{N^1_i\,:\,i<m\}$ be $P$--symmetric systems.
Suppose that $(\bigcup\mtcl N_0)\cap(\bigcup\mtcl N_1)=X$ and that there is an isomorphism $\Psi$ between the  structures
$\la \bigcup_{i<m}N^0_i,\in, P, X, N^0_i\ra_{i<m}$ and $\la \bigcup_{i<m}N^1_i,\in,  P, X, N^1_{i}\ra_{i<m}$ fixing $X$. Then  $\mtcl N_0\cup\mtcl N_1$ is a $P$--symmetric system.
\end{lemma}

\begin{proof}
Obviously, $\mtcl N_0\cup\mtcl N_1$ is a finite set of countable elementary substructures of $(H(\kappa), \in, P)$. We will check that this union also satisfies clauses  $(B)$, $(C)$ and  $(D)$ of Definition \ref{symm}. 

We start with clause $(B)$. We must show that if $i_0$, $i_1< m$ are such that $\d_{N^0_{i_0}}=\d_{N^1_{i_1}}$, then the isomorphism
$\Psi_{N^0_{i_0}, N^1_{i_1}}:=\Psi\circ\Psi_{N^0_{i_0}, N^0_{i_1}}$ fixes $N^0_{i_0}\cap N^1_{i_1}$. 
Now, if $x\in N^0_{i_0}\cap N^1_{i_1}$, then $x\in X\cap N^0_{i_0}$, which implies that $\Psi(x)=x\in N^1_{i_0}\cap N^1_{i_1}$ as $\Psi$ is an isomorphism between the structures $\la\bigcup_{i<m}N^0_i,\in, X, N^0_i\ra_{i<m}$ and $\la \bigcup_{i<m}N^1_i,\in, X, N^1_i\ra_{i<m}$. But then $x \in N^0_{i_0}\cap N^0_{i_1}$ again by the fact that $\Psi$ is an isomorphism between $\la\bigcup_{i<m}N^0_i,\in, X, N^0_i\ra_{i<m}$ and $\la \bigcup_{i<m}N^1_i,\in, X, N^1_i\ra_{i<m}$, which implies that $\Psi_{N^0_{i_0}, N^0_{i_1}}(x)=x$ and hence that $((\Psi\restr N^0_{i_1}) \circ\Psi_{N^0_{i_0}, N^0_{i_1}})(x) = \Psi_{N^0_{i_0}, N^1_{i_1}}(x)=x$. 

As to $(C)$, it suffices to note that the existence of $\Psi$ implies that $\{\delta_{N^0_i}\,:\,i<m\}= \{\d_{N^1_i}\,:\,i<m\}$.

Finally, we check that $\mtcl N_0\cup\mtcl N_1$ satisfies $(D)$ of Definition \ref{symm}. So, let $M_1$, $M_2$, $M_3 \in \mtcl N_0\cup\mtcl N_1$ be such that $M_2 \in M_1$ and $\delta_{M_1} = \delta_{M_3}$. We must verify that $\Psi_{M_1, M_3}(M_2)$ is also in $\mtcl N_0\cup\mtcl N_1$. Without loss of generality we may assume that there are indices $i$, $j \in \{1,2,3\}$ such that $M_i \in \mtcl N_0$  and $M_j \in \mtcl N_1$ (otherwise, the proof follows from the fact that $\mtcl N_0$ and $\mtcl N_1$ satisfy clause $(D)$). The case when $M_1$ and $M_3$ are both in $\mtcl N_1$ and $M_2$ is in $\mtcl N_0$ can be treated as follows. First note that there exist some $i_1$ and $i_2$ such that $M_1= N^1_{i_1}$ and $M_2 = N^{0}_{i_2}$. As $M_2 \in M_1$ and $\Psi$ is an isomorphism fixing $X$ (in particular, $M_2$), $N^{0}_{i_2} \in N^{0}_{i_1}$. But $\Psi_{N^0_{i_2}, N^1_{i_2}}=\Psi_{N^0_{i_1}, N^1_{i_1}} \restr N^0_{i_2}$ and this isomorphism also fixes $M_2$. So, $M_1$, $M_2= N^{1}_{i_2}$ and $M_3$ are elements of $\mtcl N_1$. The last case that needs to be considered is when $M_3$ is in  $\mtcl N_0$ and $M_1$ is in  $\mtcl N_1$. Just as before, we can ensure the existence of $i_1$, $i_2$ and $i_3$ such that $M_3 = N^{0}_{i_3}$, $M_1=N^1_{i_1}$ and $M_2 =N^{1}_{i_2}$. Let $i_4$ be such that $N^1_{i_4}=\Psi_{N^1_{i_1}, N^1_{i_3}}(N^1_{i_2})$ (recall that $\mtcl N_1$ satisfies clause $(D)$ of Definition \ref{symm}) and note that $N^0_{i_4}= \Psi_{M_1,M_3}(M_2)$.
\end{proof}

\section{Finite support forcing iterations with symmetric systems as side conditions}\label{a_general_construction}
In this section we describe a general construction of a  $\k$--sequence $\la \mtcl P_\a:\a\leq\k\ra$ of partial orders. This construction will depend on a predicate $P\sub H(\k)$, to be fixed at the outset, together with both a sequence $\la\dot{\mtcl Q}_\a\,:\,\a<\k\ra$ such that each $\dot{\mtcl Q}_\a$ is a $\mtcl P_\a$--name for a poset on $\k$ and a sequence $\la\dot{\mtcl R}_\a\,:\,\a<\k\ra$ such that each $\dot{\mtcl R}_\a$ is a $\mtcl P_\a$--name for a relation included in $([H(\k)]^{\al_0})^V\times \dot{\mtcl Q}_\a$ satisfying a certain closure property.\footnote{In our application, in Section \ref{the_forcing_construction}, $\dot{\mtcl R}_\a$ will be a name for a relation holding about a typical pair $(N, \n)$ exactly when $\n$ is an $(N[\dot G_\a],\, \dot{\mtcl Q}_\a)$--generic condition.}  As in Section \ref{symmetric_systems}, $\k$ can be taken in this section to be any cardinal at least $\o_2$. 

Let $P\sub H(\k)$ be a fixed predicate. We are going to describe what it means for a sequence $\la (\mtcl P_\a, \leq_\a)\,:\,\a\leq\k\ra$ of partial orders to be  \emph{the finite support iteration with $P$--symmetric systems as side conditions based on $\la (\dot{\mtcl Q}_\a, \dot{\mtcl R}_\a) \,:\,\a<\k\ra$}. This description will stretch up until Subsection \ref{general_facts}. As we will see, the description will specify one unique object (for fixed parameters $P$ and $\la (\dot{\mtcl Q}_\a, \dot{\mtcl R}_\a) \,:\,\a<\k\ra$). Hence our use above of the article `the'.


To start with, $\mtcl P_0$ must consist of all pairs of the form  

\begin{itemize}
\it[(a)] $(\emptyset, \{(N_i, 0)\,:\, i < m\})$, where $\{N_i\,:\,i<m\}$ is a $P$--symmetric system.
\end{itemize}

Given $\mtcl P_0$--conditions $q_\e=(\emptyset, \{(N^\e_i, 0)\,:\,i<m_\e\})$ for $\e\in\{0, 1\}$, $q_1 \leq_0 q_0$ if $\{ N^0_i \,:\,i<m_0\}\sub \{N^1_i \,:\,i<m_1\}$.

In the definition of $\mtcl P_0$--condition we have used the empty set in a completely vacuous way. These (vacuous)
$\emptyset$'s are there to ensure that the (uniformly defined) operation of restricting a condition in a
(further) $\mtcl P_\a$ to an ordinal $\b<\a$ yields a condition in $\mtcl P_\b$ when applied to any condition in any $\mtcl P_\a$ and to $\b=0$.

\begin{notation} If $q$ is an ordered pair, we denote the first component of $q$ by $F_q$ and the second component of $q$ by $\D_q$. Also, if $q$ is an ordered pair such that $F_q$ is a function and $\D_q$ is a relation and $\x$ is an ordinal, \emph{the restriction of $q$ to $\x$}, denoted by $q\av_\x$, is defined
as the pair $$q\av_\x:=(F_q\restr\x,\,\{(N, min\{\b,\,\x\})\,:\,(N, \b)\in \D_q\})$$ 
\end{notation} 

Let $\a\leq \k$, $\a>0$, and suppose that we have defined $\mtcl P_\x$ for all $\x<\a$. Suppose also
that if $\x<\a$ and $q \in \mtcl P_\x$, then $q$ is an ordered pair of the form $(F_q, \Delta_q)$, where

\begin{itemize}

\it $F_q$ is a finite function with domain included in $\x$, and

\it $\Delta_q$  is a finite relation $\{(N_i, \tau_i) :i \in n \}$ such that $dom(\D_q)=\{N_i\,:\,i<n\}$ is a $P$--symmetric system and, for all $i$, $\tau_i$ is an ordinal such that $\tau_i\leq\x$.
\end{itemize}


If $\a= \sigma +1$, we require that the following holds:

\begin{itemize} \it[(1)] $\dot{\mtcl Q}_{\sigma}$ is a $\mtcl P_{\sigma}$--name for a partial order.

\it[(2)] $\dot{\mathcal R}_{\sigma}$ is a $\mtcl P_{\sigma}$--name for a relation $R$ such that $$R\sub ([H(\k)]^{\al_0})^V\times\dot{\mtcl Q}_{\sigma}$$ and such that $(N, x')\in R$ whenever $(N, x)\in R$ and $x'$ is a $\dot{\mtcl Q}_{\sigma}$--condition extending $x$ (\emph{$R$ is downward closed with respect to $\dot{\mtcl Q}_{\sigma}$}).
\end{itemize}

The definition of $\mtcl P_{\a}$ is as follows (regardless of whether $\a$ is a successor or a limit ordinal). Conditions in $\mtcl P_{\a}$ are pairs of the form $$q= (F_q,\, \D_q)$$ with the following properties.

\begin{itemize}
\it[$(b\,0)$] $F_q$ is a finite function with $dom(F_q)\sub\a$.

\it[$(b\,1)$]  $\D_q$ is of the form $\{(N_i, \b_i)\,:\, i< m\}$ where, for all $i< m$, $\b_i\leq \alpha \cap sup(N_i\cap \k)$.\footnote{Note that $\b_i$ is always less than $\k$ (even when $\a=\k$).}

\it[$(b\,2)$] For all $\x<\a$, the restriction of $q$ to $\x$ is a condition in $\mtcl P_\x$. 

\it[$(b\,3)$] If $\x\in dom(F_q)$, then $F_q(\x)$ is a $\mtcl P_\x$--name and $q\av_\x\Vdash_\x F_q(\x)\in \dot{\mtcl Q}_\x$.

\it[$(b\,4)$] If $\x\in dom(F_q)$,  $(N, \b)\in \D_q$, $\b\geq\x+1$, and  $\x\in N$, then $$q\av_\x \Vdash_{\x} (N, F_q(\x))\in\dot{\mtcl R}_\x$$ 

\end{itemize}

Given conditions $$q^\e =(F_\e, \,\{(N^\e_i, \b^\e_i)\,:\, i < m_\e\})$$ (for $\e\in\{0, 1\}$) in $\mtcl P_\a$, we will say that $q^{1} \leq_{\a} q^{0}$ if and only if the following holds.

\begin{itemize}

\it[$(c\,1)$] For all $\x<\a$, $q^1\av_\x \leq_\x q^0\av_\x$.

\it[$(c\,2)$] $dom(F_0)\sub dom(F_1)$ and, for all $\x\in dom(F_0)$, $$q^1\av_\x\Vdash_{\x} F_1(\x)\leq_{\dot{\mtcl Q}_\x} F_0(\x)$$

\it[$(c\,3)$] For all $i < m_{0}$ there is some $\widetilde{\b}_i \geq \b^0_i$ such that $(N^0_{i}, \widetilde{\b}_i) \in \D_{q^{1}}$.

\end{itemize}













\begin{notation}
Given $\a\leq\k$ and a $\mtcl P_\a$--condition $q=(F_q, \D_q)$,  $dom(F_q)$ will also be denoted by $supp(q)$ and will be called the \textit{support of $q$}.
\end{notation}

Note that if $\a<\b\leq \k$, then $\mtcl P_\a\sub\mtcl P_\b$ and every  $\mtcl P_\b$--condition $q=(F, \{(N_j, \b_j)\,:\,j<m\})$ such that $supp(q)\sub \a$ and $\b_j \leq \a$ for all $j$ is also a  $\mtcl P_\a$--condition and is in fact equal to its restriction to $\a$. 

Also note that if $\a$ is a nonzero limit ordinal, then a pair $q=(F_q, \D_q)$ is a $\mtcl P_\a$--condition if and only if it satisfies ($b\,0$)--($b\,2$).

We will sometimes talk about `finite support iterations with $P$--symmetric systems as side conditions' in contexts where the sequence $\la (\dot{\mtcl Q}_\a, \dot{\mtcl R}_\a) \,:\,\a<\k\ra$ is irrelevant. We may also omit $P$ when it is not relevant.

\subsection{General facts}\label{general_facts}

In this subsection we present several facts that all finite support iterations with $P$--symmetric systems as side conditions will satisfy. 

In arguments involving this type of construction (for example in the proof of Theorem \ref{mainthm}) one naturally finds oneself having to prove that natural amalgamations of conditions are themselves conditions. The following five lemmas give some basic properties of such amalgamations that are often used in those types of arguments. 

We start with our first amalgamation lemma. An immediate consequence of this lemma (Corollary \ref{compll1}) is that our use of the term ``iteration'' is appropriate; more specifically, it follows from the lemma that  if $\la\mtcl P_\x\,:\,\x\leq\k\ra$ is a finite support iteration with $P$--symmetric systems as side conditions, then it is a forcing iteration in the sense that $\mtcl P_\a$ is a complete suborder of 
$\mtcl P_\b$ whenever $\a<\b$. 

\begin{lemma}\label{compll}
Let $P\sub H(\k)$ and let $\la\mtcl P_\a\,:\,\a\leq\k\ra$ be the finite support iteration with $P$--symmetric systems as side conditions based on $\la (\dot{\mtcl Q}_\a, \dot{\mtcl R}_\a) \,:\,\a<\k\ra$. Let $\a \leq \b \leq \k$. If $q=(F_q, \D_{q})\in \mtcl P_\a$, $r=(F_r, \D_{r}) \in \mtcl P_\b$, and $q \leq_\a
r|_\a$, then $$r \wedge_\a q:=(F_q\cup(F_r\restr [\a,\, \b)), \D_{q}\cup \D_{r})$$ is a condition in $\mtcl P_\b$ extending $r$.
\end{lemma}

\begin{proof}
The proof is a mechanical verification and proceeds by induction on $\b\geq\a$. The crucial point is the use of the markers $\b_i$ in the definition of the forcing. New side conditions $(N_i, \b_i)$ appearing in $\D_q$ may well have the property that $N_i\cap[\a,\,\b)\neq\emptyset$, but they will not impose any problematic promises -- coming from clause $(b\,4)$ in  the definition -- on ordinals $\x$ occurring in $dom(F_r\restr[\a,\,\b))$. The reason is simply that $\b_i\leq\a$. The details of the proof are as follows.

Note that the case $\beta = \alpha$ is obvious, so let us start by assuming that $\beta$ is the successor of $\sigma$ with $\sigma \geq \alpha$. Clearly, $r \wedge_\a q$ satisfies clauses $(b\,0)$ and $(b\,1)$ in the definition of $\mtcl P_{\s+1}$. By the inductive hypothesis we know that the restriction of $r \wedge_\a q$ to $\sigma$, that is, $$(r\wedge_\a q)|_\s= (F_q\cup(F_r\restr [\a,\, \s)), \D_{q}\cup \D_{r\av_\s})$$ is a condition in $\mtcl P_\s$ extending $r\av_\s$. Therefore, $r \wedge_\a q$ also satisfies $(b\,2)$. If $\s\notin dom(F_r)$, then $r \wedge_\a q$ is a condition in $\mtcl P_{\s+1}$, since clause $(b\,4)$ is automatically satisfied. If $\s\in dom(F_r)$, then $(r\wedge_\a q)\av_\s$ forces in $\mtcl P_\s$ that $F_r(\s)$ is in $\dot{\mtcl Q}_\s$ (since $r\av_\s$ forces this and $(r\wedge_\a q)\av_\s$ extends $r\av_\s$). This concludes the verification of $(b\,3)$ for $q\wedge_\a r$. Now we check that $$(r \wedge_\a q)\av_\s\Vdash_\s (N, F_r(\sigma))\in\dot{\mtcl R}_\s$$  for all those $N$ such that $(N, \sigma+1) \in \D_{q}\cup \D_{r}$ and $\s+1\in N$. But such an $N$ is such that $(N, \sigma+1) \in \D_{r}$. Since $r$ satisfies $(b\,4)$, $r\av_\s$ (and hence, the restriction of $r \wedge_\a q$ to $\sigma$) forces that $(N, F_r(\sigma))$ is in $\dot{\mtcl R}_\s$. Finally note that the inductive hypothesis
and the inclusion $\D_r \subseteq \D_{r \wedge_\a q}$ together imply that $r \wedge_\a q$ extends $r$. The case when $\beta$ is a nonzero limit ordinal follows directly from the inductive hypothesis.
\end{proof}

\begin{corollary}\label{compll1}
For every finite support iteration $\la\mtcl P_\a\,:\,\a\leq\k\ra$ with symmetric systems as side conditions and for all $\a<\b\leq\k$, every maximal antichain in $\mtcl P_\a$ is a maximal antichain in $\mtcl P_\b$, and therefore $\mtcl P_\a$ is a complete suborder of $\mtcl P_\b$.
\end{corollary}


\begin{lemma}\label{amalg1}
Let $P\sub H(\k)$ and let $\la\mtcl P_\a\,:\,\a\leq\k\ra$ be the finite support iteration with $P$--symmetric systems as side conditions based on $\la (\dot{\mtcl Q}_\a, \dot{\mtcl R}_\a) \,:\,\a<\k\ra$. Let $q_1=(F_1, \, \D_1)$ and $q_2=(F_2, \, \D_2)$ be conditions in $\mtcl P_{\a+1}$ such that there is a $\mtcl P_\a$--name $\dot x$, a condition $r=(F_r, \D_r)$ in $\mtcl P_{\a}$, and a finite set $\{M_j : j \in  n\}$ with the following properties:
\begin{itemize}
\it[(a)] $\alpha+1 \leq sup(M_j \cap \k)$ and $(M_j,\a) \in \D_r$ for all $j < n$,
\it[(b)] $r$ extends both $q_1\av_\a$ and $q_2\av_\a$,
\it[(c)] $\a\in dom(F_1)\cap dom(F_2)$ and $r$ forces in $\mtcl P_\a$ that $\dot x$ extends both $F_1(\a)$ and $F_{2}(\a)$ in $\dot{\mtcl Q}_\alpha$, and
\it[(d)] $r\Vdash_\a (M_j, \dot x)\in\dot{\mtcl R}_\a$ for all $j < n$ such that $\a\in M_j$.

\end{itemize}
Then, $$q_3=(F_r\cup\{\la \a, \dot x \ra\}, \, \D_{r} \cup \D_{1} \cup \D_{2} \cup \{(M_j, \a+1) : j \in  n\})$$ is a condition in $\mtcl P_{\a+1}$ extending both $q_1$ and $q_2$.
\end{lemma}



\begin{proof}
First we check that $q_3$ is in $\mtcl P_{\a+1}$. It follows from $(a)$ and $(b)$ that the restriction of $q_3$ to $\a$ is equal to $r$, and hence that $q_3$ satisfies clauses $(b\,0)$--$(b\,2)$. Condition $(b\,3)$ for $q_3$ follows from (c).
Finally we must show that $r$ forces  $(N, \dot x)\in\dot{\mtcl R}_\a$ for all those $N$ such that $(N, \a+1)$ is in $\D_{1}\cup \D_{2} \cup \{(M_j, \a+1) : j \in  n\}$ and such that $\a \in N$ (recall that if $(N,\g) \in \D_r$, then $\gamma \leq \alpha$). But for such an $N$, if $(N, \a+1) \in \D_{i}$ ($i \in \{1,2\}$), it suffices to recall that $r \leq_\a q_i\av_\a$, that $r$ forces  $\dot x \leq_{\dot{\mtcl Q}_\alpha} F_i(\a)$, and that (by clause $(b\,4)$ applied to $q_i$) $q_i\av_\a$ forces $(N, F_i(\a))\in\dot{\mtcl R}_\a$. Hence, $r$ forces $(N, \dot x)\in\dot{\mtcl R}_\a$ by the downward closure of  $\dot{\mtcl R}_\a$ with respect to $\dot{\mtcl Q}_\a$. The case when $(N, \a+1) \in \{(M_j, \a+1) : j \in  n\}$ follows from $(d)$. Finally note that $(b)$, $(c)$ and the inclusion $\D_i \subseteq \D_{1} \cup \D_{2} \cup \{(M_j, \a+1) : j \in  n\}$ imply together that $q_3$ extends $q_i$ for $i\in\{1, 2\}$.
\end{proof}

Exactly the same proof establishes the following variant of Lemma \ref{amalg1}.

\begin{lemma}\label{amalg1.5}
Let $P\sub H(\k)$ and let $\la\mtcl P_\a\,:\,\a\leq\k\ra$ be the finite support iteration with $P$--symmetric systems as side conditions based on $\la (\dot{\mtcl Q}_\a, \dot{\mtcl R}_\a) \,:\,\a<\k\ra$. Let $q_1=(F_1, \, \D_1)$ and $q_2=(F_2, \, \D_2)$ be conditions in $\mtcl P_{\a+1}$, $r=(F_r, \D_r)$ a condition in $\mtcl P_{\a}$, and $\{M_j : j \in  n\}$ a finite set with the following properties:
\begin{itemize}
\it[(a)] $\alpha+1 \leq sup(M_j \cap \k)$ and $(M_j,\a) \in \D_r$ for all $j < n$,
\it[(b)] $r$ extends both $q_1\av_\a$ and $q_2\av_\a$, and
\it[(c)] $\a\notin dom(F_1)\cup dom(F_2)$.
\end{itemize}
Then, $$q_3=(F_r, \, \D_{r} \cup \D_{1} \cup \D_{2} \cup \{(M_j, \a+1) : j \in  n\})$$ is a condition in $\mtcl P_{\a+1}$ extending both $q_1$ and $q_2$.

Suppose, in addition, that $\dot x$ is a $\mtcl P_\a$--name  such that \begin{itemize}

\it[(d)] $r\Vdash_\a (M_j, \dot x)\in\dot{\mtcl R}_\a$ for all $j < n$ such that $\a\in M_j$, and \it[(e)] $r\Vdash_\a (N, \dot x)\in\dot{\mtcl R}_\a$ for all $N$ such that $(N, \a+1)\in \D_1\cup\D_2$ and $\a\in N$.\end{itemize} 

Then, $$q'_3=(F_r\cup\{\la \a, \dot x \ra\}, \, \D_{r} \cup \D_{1} \cup \D_{2} \cup \{(M_j, \a+1) : j \in  n\})$$ is a condition in $\mtcl P_{\a+1}$ extending both $q_1$ and $q_2$.
\end{lemma}


\begin{lemma}\label{amalg2}
Let $P\sub H(\k)$ and let  $\la\mtcl P_\a\,:\,\a\leq\k\ra$ be the finite support iteration with $P$--symmetric systems as side conditions. Assume that $0\leq \sigma < \alpha \leq \kappa$. Let $q_\x=(F_\x,\D_\x)$ $(\x \in \{ 1,2 \})$ be  conditions in $\mtcl P_\a$ such that $supp(q_1)\cup supp(q_2) \subseteq \sigma$ and such that there exists a condition $r =(F_r, \D_r) \in \mtcl P_{\s}$ extending both $q_1\av_\s$ and $q_2\av_\s$. Then $q_1$ and $q_2$ are compatible in $\mtcl P_\a$.
\end{lemma}

\begin{proof}
Define $q_3=(F_r, \D_r\cup \D_{1} \cup \D_{2})$. We prove by induction on $\b$, $\s\leq \b\leq \a$, that $q_3\av_\b$ is a condition in $\mtcl P_\b$ extending $q_1\av_\b$ and $q_2\av_\b$. The successor step follows from Lemma \ref{amalg1.5}.
\end{proof}

\begin{lemma}\label{amalg3}
Let $P\sub H(\k)$ and let $\la\mtcl P_\a\,:\,\a\leq\k\ra$ be the finite support iteration with $P$--symmetric systems as side conditions based on $\la (\dot{\mtcl Q}_\a, \dot{\mtcl R}_\a) \,:\,\a<\k\ra$. Let $0< \b \leq \kappa$. Given conditions $q_\x=(F_\x,\D_\x)$ $(\x \in \{ 0,1 \})$ in $\mtcl P_\b$, let $Z_{\x} = supp(q_\x)\cup(\b \cap \bigcup dom(\D_{q_\x}))$. Let $\a\leq \b$ be such that $Z_{0} \cap Z_{1} \subseteq \a$, and assume there is a condition $r=(F_r, \D_r)$ in $\mtcl P_\a$ extending $q_0 \av_\a$ and $q_1 \av_\a$. Let $F_r^{0,1} = F_r\cup (F_0\restr[\a,\,\b))\cup (F_1\restr[\a,\,\b))$. 
Then the natural amalgamation of $r$, $q_0$ and $q_1$, i.e.,
$$(q_0 \wedge q_1) \wedge_\a r:= (F_r^{0,1}, \D_r \cup \D_0\cup \D_1)$$ is a $\mtcl P_\b$--condition extending $q_0 $ and $q_1$.
\end{lemma}

\begin{proof}
The proof is by induction on $\beta\geq \alpha$. Note that the case $\beta = \alpha$ is clear, so let us start by assuming that $\beta$ is the successor of $\sigma$ with $\sigma \geq \alpha$. Clearly, $(q_0 \wedge q_1) \wedge_\a r$ satisfies clauses $(b\,0)$ and $(b\,1)$. Using the inductive hypothesis we know that the restriction of the amalgamation to $\sigma$ is a condition in
$\mtcl P_\s$ which extends both $q_0\av_\s$ and $q_1\av_\s$. In particular, if $\s\in dom(F^{0, 1}_r)$, then $((q_0 \wedge q_1) \wedge_\a r)\av_\s$ forces $F^{0, 1}_r(\s)\in\dot{\mtcl Q}_\s$. Therefore, $ (q_0 \wedge q_1) \wedge_\a r$ also satisfies $(b\,2)$ and $(b\,3)$.
Let us assume now  that $\sigma \in dom(F^{0, 1}_r)$. In fact, since $supp(q_0) \cap supp(q_1) \subseteq \alpha$ and $\sigma \geq \alpha$, we may assume that $\sigma \in supp(q_0) \setminus supp(q_1)$ (the proof when $\sigma \in supp(q_1)$ is identical). We must show that the condition $((q_0 \wedge q_1) \wedge_\a r)\av_\s$ forces $(N, F_r^{0,1}(\sigma))\in\dot{\mtcl R}_\s$ for all those $N$ such that $(N, \sigma+1) \in \D_{0}\cup \D_{1}$ and $\s\in N$.  Since $\s \geq \alpha$ and $Z_{0} \cap Z_{1} \subseteq \a$, it follows for such an $N$ that $(N, \sigma+1) \in \D_0$. Since $q_0$ satisfies $(b\,4)$ and  $F_r^{0,1}(\sigma)= F_0(\sigma)$, $q_0\av_\s$ (and hence, $((q_0 \wedge q_1) \wedge_\a r)\av_\s$) forces $(N, F_r^{0,1}(\sigma))\in\dot{\mtcl R}_\s$.  Finally note that the inductive hypothesis, the choice of
$F_r^{0,1}$ and the inclusion $\D_0 \cup \D_1 \subseteq \D_{(q_0 \wedge q_1) \wedge_\a r}$ together imply that $(q_0 \wedge q_1) \wedge_\a r$ extends $q_0$ and $q_1$. The case when $\beta$ is a nonzero limit ordinal follows directly from the inductive hypothesis.
\end{proof}

The final result in this section applies, assuming $\textsc{CH}$, to iterations with $P$--symmetric systems as side conditions for $P$ for which there is  a bijection $\varphi:H(\k)\into \k$ definable in $(H(\k), \in, P)$ and such that, moreover, each $\dot{\mtcl Q}_\a$ is forced to be a poset on $\o_1^V$.  It shows that all such iterations have the $\al_2$--chain condition. We will actually show that these forcings are $\al_2$--Knaster.\footnote{A forcing $\mtbb P$ is $\mu$--$Knaster$ if every
subset of $\mtbb P$ of cardinality $\mu$ includes a subset of cardinality
$\mu$ of pairwise compatible conditions.}

\begin{lemma}\label{cc} ($\textsc{CH}$) Let $\la\mtcl P_\a\,:\,\a\leq\k\ra$ be the finite support iteration with $P$--symmetric systems as side conditions based on  $\la (\dot{\mtcl Q}_\a, \dot{\mtcl R}_\a) \,:\,\a<\k\ra$. Suppose that \begin{itemize} \it there is  a bijection $\varphi:H(\k)\into \k$ definable in $(H(\k), \in, P)$, and that \it for all $\a<\k$, $\mtcl P_\a$ forces $\dot{\mtcl Q}_\a\sub\o_1^V$. \end{itemize}

Then for every ordinal $\a\leq \k$, $\mtcl P_\a$ is $\al_2$--Knaster.
\end{lemma}

\begin{proof}
The proof is by induction on $\a$ and involves standard $\D$--system and pigeonhole principle arguments. The conclusion for $\a=0$ follows from $\textsc{CH}$: Suppose $m<\o$ and $q_\x=\{N^\x_i\,:\,i<m\}$ is a $\mtcl P_0$--condition for each $\x<\o_2$. For notational convenience we are identifying a $\mtcl P_0$--condition $q$ with $dom(\D_q)$, which is fine for this proof. By $\textsc{CH}$ we may assume that $\{\bigcup_{i<m}N^\x_i\,:\,\x<\o_2\}$ forms a $\D$--system with root $X$. Furthermore, by $\textsc{CH}$ we may assume, for all $\x$, $\x'<\o_2$, that the structures
$\la \bigcup_{i<m}N^\x_i,\in, P,  X, N^\x_i\ra_{i<m}$ and $\la \bigcup_{i<m}N^{\x'}_i,\in, P,  X, N^{\x'}_{i}\ra_{i<m}$  are isomorphic and that the corresponding isomorphism fixes $X$. The first assertion follows from the fact that there are only $\al_1$--many isomorphism types for such structures. For the second assertion note that, if $\Psi$ is  the unique isomorphism between $\la \bigcup_{i<m}N^\x_i,\in, P,  X, N^\x_i\ra_{i<m}$ and $\la \bigcup_{i<m}N^{\x'}_i,\in,  P, X, N^{\x'}_{i}\ra_{i<m}$, then the restriction of $\Psi$ to $X\cap\k$ has to be the identity on $X\cap\k$. Since  there is  a bijection $\varphi:H(\k)\into \k$ definable in $(H(\k), \in, P)$, we have that $\Psi$ fixes $X$ if and only if it fixes $X\cap\k$. It follows that $\Psi$ fixes $X$.
Hence, by Lemma \ref{iso3} we have, for all $\x$, $\x'<\o_2$, that $q_\x\cup q_{\x'}$ extends both $q_\x$ and $q_{\x'}$.

For $\a=\s+1$, suppose $q_\x$ is a $\mtcl P_{\s+1}$--condition for each $\x<\o_2$. Suppose, without loss of generality, that each $q_\x$ is of the form $q_\x=(F_\x,\, \D_\x)$ with $\sigma \in dom(F_\x)$ (the proof in the case that there are $\al_2$--many $q_\x$ of the form $(F, \D)$ with $dom(F)\sub\s$ follows directly from Lemma \ref{amalg1.5}). By extending $q_\x$ if necessary, we may assume that each $F_\xi(\sigma)$ is the canonical $\mtcl P_\s$--name for an actual ordinal in $\o_1$. We may also assume by our induction hypothesis that all $q_\x\av_\s$ are pairwise compatible. Let $r_{\x, \x'}\in\mtcl P_\s$ be a condition extending both $q_\x\av_\s$ and $q_{\x'}\av_\s$ for all $\x<\x'<\o_2$. Now find a set $I\sub\o_2$ of size $\al_2$ such that for all $\x<\x'$ in $I$, $F_\xi(\sigma)=F_{\x'}(\sigma)$. By Lemma \ref{amalg1} it follows now, for all such $\x$, $\x'$, that the natural amalgamation of $r_{\x, \x'}$, $q_\x$ and $q_{\x'}$ is a $\mtcl P_{\s+1}$--condition extending $q_\x$ and $q_{\x'}$.

For $\a$ a nonzero limit ordinal, suppose $q_\x$ is a $\mtcl P_\a$--condition for all $\x<\o_2$. Suppose first that $cf(\a) \neq \o_2$. There is then some $\s<\a$ such that $I=\{\x<\o_2\,:\,supp(q_\x)\sub\s\}$ has size $\al_2$. By induction hypothesis there is some $I'\sub I$ of size $\al_2$ such that all $q_\x \av_\s$ (for $\x\in I'$) are pairwise compatible in $\mtcl P_\s$. But now it follows from Lemma \ref{amalg2} that $q_\x$ and $q_{\x'}$ are compatible in $\mtcl P_\a$ for all $\x<\x'$ in $I'$.

Finally, suppose $cf(\a)=\o_2$. For each $\x < \o_2$, let $Z_{\x}$ be equal to the union of the sets $supp(q_\x)$ and $\a \cap \bigcup dom(\D_{q_\x})$. By $\textsc{CH}$ we may find $I\sub\o_2$ of size $\al_2$ such that $\{Z_{\x} \,:\,\x\in I\}$ forms a $\D$--system with root $X$. 

Let now $\s<\a$ be such that $X\sub \s$ ($\s$ exists by $cf(\a)\geq\o_1$). By induction hypothesis we may assume that all $q_\x\av_\s$ are pairwise compatible in $\mtcl P_\s$. For all $\x<\x'$ in $I$ let $r_{\x, \x'}$ be a condition in $\mtcl P_\s$ extending $q_\x\av_\s$ and $q_{\x'}\av_\s$. From Lemma \ref{amalg3} it follows then, for all such $\x$, $\x'$, that the natural amalgamation of $r_{\x, \x'}$, $q_\x$ and $q_{\x'}$ is a $\mtcl P_\a$--condition extending $q_\x$ and $q_{\x'}$.
\end{proof}

\section{Proving Theorem \ref{mainthm}}\label{the_forcing_construction}

We will now proceed to the definition of the forcing $\mtcl P$ witnessing Theorem \ref{mainthm}.  The proof of Theorem \ref{mainthm} will then be given in a sequence of lemmas.

For this section, assume $\textsc{CH}$ holds and let us fix a cardinal $\k$ such that $\k^{\al_1}= \k$ and $2^{<\k}=\k$. Let $\Phi:\k\into H(\k)$ be a surjection such that for every $x$ in $H(\kappa)$, $\Phi^{-1}(\{x\})$ is unbounded in $\kappa$. Let also $\lhd$ be a well--order of $H(\k^+)$ in order type $2^\k$. The bookkeeping function $\Phi$ exists by  $2^{<\k}=\k$, and $\lhd$ exists since $\av H(\k^+)\av = 2^\k$.

Let $\la \t_\a\,:\,\a\leq \k\ra$ be the strictly increasing sequence of regular cardinals defined as $\t_0 =\av 2^{\k}\av^{+}$ and $\t_\a=|2^{sup\{\t_\beta\,:\,\b\leq \a \}}|^{+}$ if $\a>0$. For each $\a\leq\k$ let $\mtcl M^\ast_\a$ be the collection of all countable elementary substructures of $H(\t_\a)$ containing $\Phi$, $\lhd$ and $\la \t_\b\,:\,\b<\a\ra$. Let also $\mtcl M_\a=\{N^\ast\cap H(\k)\,:\,N^\ast\in\mtcl M^\ast_\a\}$ and note that if $\alpha < \beta$, then $\mtcl M^\ast_\a$ belongs to all members of $\mtcl M^\ast_\b$ containing the ordinal $\alpha$.

The forcing $\mtcl P$ witnessing Theorem \ref{mainthm}  will be $\mtcl P_\k$, where the sequence $\la\mtcl P_\a\,:\,\a\leq\k\ra$ is the finite support iteration with $\Phi$--symmetric systems as side conditions based on a certain sequence $\la(\dot{\mtcl Q}_\a, \dot{\mtcl R}_\a)\,:\,\a<\k\ra$ of pairs of names. Let $\a<\k$ be given and suppose $\mtcl P_\a$ has been defined.

In Definition \ref{fin_proper}, and throughout the rest of the paper, we will abuse notation slightly when writing things like $N[H]$: 
Given a set $N$, a partial order $\mtbb P$ and a filter $H\sub\mtbb P$, $N[H] = \{\tau_H\,:\,\tau\in N\mbox{ is a }\mtbb P\mbox{--name}\}$.
Note that we are not requiring that $\mtbb P$ be in $N$, or even that $\mtbb P\cap N$ be a definable class in $N$.

\begin{definition}\label{fin_proper}
(For $\a<\k$, in $V[G_\a]$, where $G_\a$ is a $\mtcl P_{\a}$-generic filter) Let $\mtcl Q$ be a forcing on $\omega^V_1$.
We will say that \emph{$\mtcl Q$ is $V$--finitely proper (with respect to $G_\a$)} if and only if there exists a club $D \subseteq ([H(\kappa)]^{\aleph_0})^V$ in $V$ with the following property:

If $m< \omega$ and $\{N_i \,:\, i \in m\}\sub D$ is such that $\{(N_i, \a)\,:
\,i<m\} \sub \D_u$ for some $u \in G_\alpha$ and such that $\mtcl Q \in N_i[G_\alpha]$ for all $i$, then for every $\nu \in \bigcap\{N_i \cap \omega_1 \,:\, i \in m\}$ there exists some $\nu^{\ast}$ such that $\nu^{\ast}$ extends $\nu$ in $\mtcl Q$ and is $(N_i[G_\a], \mtcl Q)$--generic for all $i$.
\end{definition}

The definition of $(\dot{\mtcl Q}_\a, \dot{\mtcl R}_\a)$ is the following: $(\dot{\mtcl Q}_\a, \dot{\mtcl R}_\a)$ is the $\lhd$--least pair of $\mtcl P_\a$--names in $H(\k^+)$ with the following properties:

\begin{enumerate} 

\it If $\Phi(\alpha)= \dot{\mtcl Q}$ is a $\mtcl P_\a$--name for a nontrivial\footnote{Nontrivial in the sense that it has some condition different from $0$.} $V$--finitely proper forcing on $\omega_1^V$, then $\dot{\mtcl Q}_\a=\dot{\mtcl Q}$.\footnote{In Lemma \ref{horribilis} we will see that $\mtcl P_\a$ preserves $\o_1$.}

\it If $\Phi(\alpha)$ is not a $\mtcl P_\a$--name for a nontrivial $V$--finitely proper forcing on $\omega_1^V$, then $\dot{\mtcl Q}_\a$ is a $\mtcl P_\a$--name for trivial forcing on $\{0\}$ (which is a $V$--finitely proper forcing on $\omega_1^V$).

\it  $\dot{\mtcl R}_\a$ is a $\mtcl P_\a$--name for the set of pairs $(N, \n)$ such that \begin{enumerate}

\it $\nu\in\dot{\mtcl Q}_\a$, and

 \it if $N\in\mtcl M_{\a+1}$, then $\nu$ is $(N[\dot G_\a], \dot{\mtcl Q}_\a)$--generic.
 
 \end{enumerate} \end{enumerate}

Note that $\dot{\mtcl R}_\a$ is forced to be downward closed with respect to $\dot{\mtcl Q}_\a$, so the construction makes sense.
Note also that if $\alpha < \beta \leq \kappa$, $N^\ast\in\mtcl M^\ast_\beta$ and $\alpha \in N^\ast$, then $\mtcl P_\alpha \in N^\ast$.


We are going to prove the relevant properties of the forcings $\mtcl P_\a$. Theorem \ref{mainthm} will follow immediately from them.

The hypotheses of Lemma \ref{cc} clearly apply to our construction. Hence, every $\mtcl P_\a$ is $\al_2$--Knaster.

\begin{lemma}\label{solid} For every $\a\leq\k$ and $q\in\mtcl P_\a$ there is an extension $q'$ of $q$ such that $F_{q'}(\x)$ is a canonical $\mtcl P_\x$--name for an ordinal in $\o_1^V$ for every $\x\in supp(q')$.\end{lemma} 

The proof of Lemma \ref{solid} is a straightforward induction using the fact that conditions in $\mtcl P_\a$ have finite support.

\begin{definition}
Given $\a\leq \kappa$, a condition $q\in\mtcl P_\a$, and a countable elementary substructure $N \prec H(\kappa)$, we will say that $q$ is $(N,\, \mtcl P_\a)$--pre-generic in case
\begin{itemize}

\it $\a < \k$ and the pair $(N, \a)$ is in $\Delta_q$, or else

\it $\a = \k$ and the pair $(N, sup(N\cap \k))$ is in $\Delta_q$.

\end{itemize}
\end{definition}

The properness of all $\mtcl P_\a$ is an immediate consequence of the following lemma.

\begin{lemma}\label{horribilis}

Suppose $\a\leq \kappa$ and $N^\ast\in\mtcl M^\ast_\a$. Let $N=N^\ast\cap H(\k)$. Then the following conditions hold.

\begin{itemize}

\it[$(1)_\a$] For every $q\in N$ there is some $q'\leq_\a q$ such that $q'$ is $(N,\, \mtcl P_\a)$--pre-generic.

\it[$(2)_\a$] If $\mtcl P_\a\in N^\ast$ and $q\in\mtcl P_\a$ is $(N,\, \mtcl P_\a)$--pre-generic, then $q$ is $(N^\ast,\, \mtcl P_\a)$--generic.

\end{itemize}

\end{lemma}

\begin{proof}
The proof will be by induction on $\a$. We start with the case $\a =0$. For simplicity we are going to identify a $\mtcl P_0$--condition $q$ with $dom(\D_q)$. The proof of $(1)_0$ is trivial: It suffices to set $q'= q\cup\{N\}$.

The proof of $(2)_0$ is also easy: Let $E$ be a dense subset of $\mtcl P_0$ in $N^\ast$. It suffices to show that there is some condition in $E\cap N^\ast$ compatible with $q$. Notice that $q\cap N^\ast\in\mtcl P_0$ by Lemma \ref{iso2} (ii). Hence, we may find a condition $q^\circ\in E\cap N^\ast$ extending $q\cap N^\ast$.
Now let $$q^\ast=q\cup\{\Psi_{N, \ov N}(M) \,:\, M \in q^\circ,\,\ov N\in dom(\D_q),\,\d_{\ov N}=\d_N\}$$ By Lemma \ref{iso2} (iii) we have that $q^\ast$ is a condition in $\mtcl P_0$ extending both $q$ and $q^\circ$.

Let us proceed to the more substantial case $\a=\s+1$. We start by proving $(1)_\a$. Assume first  that $\sigma \in dom(F_q)$ and let $\n = F_q(\sigma)$. By Lemma \ref{solid} we may assume that $\n$ is the canonical name for an ordinal in $\o_1$. By $(1)_\s$ we may also assume, by extending $q\av_\s$, that $q\av_\s$ is $\mtcl P_\s$--pre-generic for $N$. Let $\dot D$ be the $\lhd$--first $\mtcl P_\s$--name for a club $D$ of $([H(\k)]^{\al_0})^V$ in $V$ such that $D$ witnesses the $V$--finite properness of $\dot{\mtcl Q}_\s$  and note that $q\av_\s$ forces $N \in \dot D$ since it forces $N^\ast[\dot G_\sigma]\cap V = N^\ast$ (which follows from $(2)_\sigma$ and from the fact that $q\av_\s$ is $\mtcl P_\s$--pre-generic for $N$). By the definition of $V$--finite properness and the fact that $(N, \s)\in \D_{q\av_\s}$, there is then some $z=(F_z, \Delta_z) \in \mtcl P_\s$ extending $q\av_\s$ and an ordinal $\nu^{\ast}$ such that $z$ forces that $\nu^\ast$ is an $(N[\dot G_\s], \dot{\mtcl Q}_\s)$--generic condition extending $\nu$. It suffices to define $q'$ as the condition $(F_z\cup\{\la\s,  \nu^{\ast} \ra\}, \D_q \cup \D_{z} \cup \{(N,\a)\})$ (Lemma \ref{amalg1} ensures that $q'$ is a condition extending $q$).

The proof in the case that $q=(F, \D)$ with $dom(F)\sub\s$ can be reduced to the previous case by the following Claim. 

\begin{claim}\label{quecosa}
If $q=(F, \D)$ and $\s\notin dom(F)$, then we can find a condition $q'=(F', \D')$ extending $q$ and such that $\s\in dom(F')$.
\end{claim}

\begin{proof}
This is true, using Lemma \ref{amalg1.5}, by essentially the same argument as above since $(2)_\sigma$ guarantees that $q\av_\s$ is also $(M^\ast,\, \mtcl P_\s)$--generic for all $M^\ast \in \mtcl M^\ast_{\s+1}$ such that $(M, \s+1)\in\D_q$ for $M=M^\ast\cap H(\k)$, which implies that a condition forcing  that all these $M$ are in $\dot D$ can be found as in that argument. \end{proof}

Now let us prove $(2)_\a$. 
Let $E$ be an open dense set of $\mtcl P_\a$ in $N^\ast$. We should find a condition $\tilde{q} \in E \cap N^\ast$ compatible with $q$. Since $E$ is dense and open, we may start assuming that $q \in E$. By Claim \ref{quecosa} we may also assume that $\sigma \in dom(F_q)$. Let $\nu = F_q(\sigma)$.  Let $G_{\sigma}$ be a $\mtcl P_{\sigma}$-generic
filter over $V$ with $q\av_{\sigma} \in G_{\sigma}$. By $(2)_\s$ we have that $G_{\sigma}$ is also generic over
$N^\ast$. Define $E/G_{\sigma}$ as the set of those conditions in $E$ whose restriction to $\s$ belongs to $G_{\sigma}$, and $\tilde{E}$ as the set of those $\eta < \o_1$ such that either

\begin{itemize}

\it[(i)] there exists some $t\in E/G_{\sigma}$ such that $\sigma\in dom(F_t)$ and $\eta = F_t(\sigma)$, or else

\it[(ii)] there is no $\eta'$ in $(\dot{\mtcl Q}_\s)_{G_\s}$ extending $\eta$ for which there is
any $t\in E/G_{\sigma}$ such that $\sigma\in dom(F_t)$ and $\eta'=F_{t}(\sigma)$.

\end{itemize}

Note that $\tilde{E}$ is a dense subset of $(\dot{\mtcl Q}_\s)_{G_{\sigma}}$ and that $\tilde{E} \in N^{\ast}[G_{\s}]$.  In fact, $\tilde{E}$ is in $N[G_\s]$ by the $\al_2$--c.c.\ of $\mtcl P_\s$ and the fact that $\dot{\mtcl R}_\s$ is a partial order on $\o_1^V$. Hence, by condition $(b\,4)$ in the definition of $\mtcl P_{\a}$ together with the choice of $\dot{\mtcl R}_\sigma$ we know that there is some $\eta \in \tilde{E}\cap N[G_{\sigma}]$ such that $\nu$ and $\eta$ are $(\dot{\mtcl Q}_\s)_{G_{\sigma}}$--compatible.

\begin{claim}\label{cl453}
Condition (i) above holds for $\eta$.
\end{claim}

\begin{proof}
Let $r$ be a condition in $G_\s$ extending $q\av_{\s}$ and let $\eta'$ be such that $r$ forces that $\eta'$ is a condition in $\dot{\mtcl Q}_\s$ extending both $\eta$ and $\nu$. But then $q^{\ast}: =(F_r\cup\{\la \s, \eta' \ra\ra\}, \D_r\cup \D_q)$ is a $\mtcl P_\a$ condition extending $q$ by Lemma \ref{amalg1}, $q^{\ast} \in E/G_{\sigma}$, and $q^{\ast}\av_{\sigma}$ forces that condition (i) holds for $\eta$ since $\eta'$ witnesses the failure of (ii) for $\eta$. This shows that $q\av_{\s}$ forces that condition (i) holds for $\eta$.
\end{proof}

By the above claim and by $N^{\ast}[G_{\sigma}]\prec H(\theta_\sigma)^V[G_\sigma]$, there is a condition $\tilde q$ in $E/G_{\sigma} \cap N^{\ast}[G_\sigma]$ such that $\sigma\in dom(F_{\tilde q})$ and $F_{\tilde q}(\sigma)=\eta$, and of course $\tilde q \in N$ since $N^\ast[G_\sigma]\cap V = N^\ast$ by $(2)_\sigma$. It remains to see that $\tilde{q}$ is compatible with $q$. For this, notice that there is some $w \in G_\sigma$ extending $q\av_\sigma$ and $\tilde q \av_\sigma$ and there is some $\eta^\ast < \omega_1$ such that $w$ forces that $\eta^\ast$ extends $\eta$ and $\nu$ in $(\dot{\mtcl Q}_\s)_{G_{\sigma}}$. But then $w$ forces that $\eta^\ast$ is $(M^\ast[\dot G_\sigma],\, (\dot{\mtcl Q}_\s)_{G_\sigma})$--generic whenever $M^\ast \in \mtcl M^\ast_\alpha$ is such that $(M^\ast \cap H(\kappa), \a)\in\Delta_{\tilde q} \cup \Delta_q$ since $\eta^\ast$ extends $\eta$ and $\nu$. It follows that $(F_w\cup\{\la \s, \eta^\ast \ra\}, \Delta_q \cup \Delta_{\tilde q} \cup \Delta_w)$ is a common extension of $q$ and $\tilde q$ by Lemma \ref{amalg1}.

It remains to prove the lemma for the case when $\a$ is a nonzero limit ordinal. We start out proving $(1)_\a$:
Let $\s\in\a\cap N$ be a bound for $supp(q)$. By $(1)_\s$ there is a condition $t\leq_\s q\av_\s$ which is pre-generic for $N$. Now let $q'=(F_t, \D_t\cup \D_q\cup\{(M, \a\cap sup(N\cap\k))\})$. It suffices to prove by induction on $\x\in[\s,\,\a]$ that $q'\av_\x$ is a condition in $\mtcl P_\x$. The limit case of the induction follows immediately from the induction hypothesis, and the successor case follows trivially from the fact that $dom(F_t)\sub \s$, and so condition $(b\,4)$ in the definition of $\mtcl P_{\x+1}$ does not apply at that stage for $q'\av_{\x+1}$.

For $(2)_\a$, let $E\sub\mtcl P_\a$ be dense and open, $E\in N^\ast$, and let $q$ satisfy the hypothesis of $(2)_\a$. We want to find a condition in $E\cap N^\ast$ compatible with $q$. We may assume that $q\in E$.

Suppose first that $cf(\a)=\o$. In this case we may take $\s\in N^\ast\cap\a$ above $supp(q)$. Let $G_\s$ be $\mtcl P_\s$--generic with $q\av_\s\in G_\s$. In $N^\ast[G_\s]$ it is true that there is a condition $q^\circ\in\mtcl P_\a$ such that

\begin{itemize}

\it[(a)] $q^\circ\in E$ and $q^\circ\av_\s\in G_\s$, and

\it[(b)] $supp(q^\circ)\sub \s$.

\end{itemize}

\noindent (the existence of such a $q^\circ$ is witnessed in $V[G_\s]$ by $q$.)

Since $q\av_\s$ is $(N^\ast, \mtcl P_\s)$--generic by induction hypothesis, $q^\circ\in N^\ast$. By extending $q$ below $\s$ if necessary, we may assume that $q\av_\s$ decides $q^\circ$ and extends $q^\circ\av_\s$. But now, the natural amalgamation $(F_q, \D_q\cup \D_{q^\circ})$ of $q$ and $q^\circ$ is a $\mtcl P_\a$--condition extending them by Lemma \ref{amalg2}.

Finally, suppose $cf(\a)\geq\o_1$. We may assume that $supp(q)$ is not bounded by $sup(N\cap\a)$ as otherwise we can argue as in the $cf(\a)=\o$ case. Thanks to Lemma \ref{solid}, by extending $q$ if necessary we may also assume that, for every $\x\in supp(q)$, $F_q(\x)$ is the canonical $\mtcl P_\x$--name for some ordinal in $\o_1$. 

Notice that if $N'\in dom(\D_q)$ and $\d_{N'}<\d_N$, then $sup(N'\cap N\cap \a)\leq sup(\Psi_{\ov N, N}(N')\cap\a)\in N\cap\a$ whenever $\ov N\in dom(\D_q)$ is such that $\d_{\ov N}=\d_N$ and $N'\in \ov N$. To see this, recall that $\Psi_{\ov N, N}$ fixes $\ov N \cap N \cap \kappa$. Also, $sup(\Psi_{\ov N, N}(N')\cap\a)\in N\cap\a$ since in $N$ it holds that $\Psi_{\ov N, N}(N')$ is countable and that $\a$ has uncountable cofinality. This is the only place in the proof where the symmetry of the systems $dom(\D_q)$ is needed. The symmetry of the systems $dom(\D_q)$ is needed precisely to derive the conclusion that $sup(N'\cap N\cap\a)<sup(N\cap\a)$ for every $N'\in dom(\D_q)$ with $\d_{N'}<\d_N$. 

Hence we may fix $\s\in N\cap \a$ such that:
\begin{itemize}

\it[(i)] $sup(N'\cap N\cap \a) < \sigma$ for all $N'\in dom(\D_q)$ with $\d_{N'}<\d_N$, and
\it[(ii)]  if $\eta \in supp(q)$ and $\eta < sup(\a \cap N)$, then $\eta < \sigma$.

\end{itemize}

As in the case $cf(\a)=\o$, if $G_\s$ is $\mtcl P_\s$--generic with $q\av_\s\in G_\s$, then in $N^\ast[G_\s]$ we can find a condition $q^\circ \in\mtcl P_\a$ such that $q^\circ\in E$, $q^\circ\av_\s\in G_\s$, $supp(q^\circ)\setminus \sigma\neq \emptyset$ and such that, for each $\x\in supp(q^\circ)$, $F_{q^\circ}(\x)$ is the canonical $\mtcl P_\x$--name for an ordinal in $\o_1$  (again, the existence of such a condition is witnessed in $V[G_\s]$ by $q$), and such a $q^\circ$ will necessarily be in $N^\ast$. By extending $q$ below $\s$ we may assume that $q\av_\s$ decides $q^\circ$ and extends $q^\circ\av_\s$. The proof of $(2)_\a$ in this case will be finished if we can show that there is a condition $q^\dag$ extending $q$ and $q^\circ$. The condition $q^\dag$ can be built by recursion on $supp(q^\circ)\setminus \sigma$ (note that by the choice of $\sigma$, $min(supp(q)\setminus \sigma)\geq sup(N\cap\alpha)$, and therefore $min(supp(q)\setminus \sigma)> max(supp(q^\circ))$). This finite construction mimics the proof of $(1)_\b$ for successor $\b$, but also uses the assumption of $V$--finite properness. The details are as follows.

Let $(\x_i)_{i<r}$ be the strictly increasing enumeration of $supp(q^\circ)\setminus \s$. Note that $r>0$, so $r-1\geq 0$. We build a sequence $(q_i)_{i<r}$ of conditions as follows:

For $i=0$, we first extend $q\av_\s$ to a $\mtcl P_{\x_0}$--condition $\ov q$ extending $q\av_{\x_0}$ and $q^\circ\av_{\x_0}$. $\ov q$ can be found by appealing to Lemma \ref{amalg2} if $\s < \x_0$, and if $\s=\x_0$ it is enough of course to set $q_0=q\av_\s$. Now note that $F_{q^{\circ}}(\x_0)=\check{\pi}$ is the canonical name for an ordinal $\pi$ in the intersection of all $\overline{N}$ with $\overline{N}\in dom(\D_q)$, $\x_0 \in \overline{N}$ and $\d_{\overline{N}}\geq\d_{N^\ast}$ (on the other hand, (i) implies that there is no $N'\in dom(\D_q)$ with $\d_{N'}<\d_N$ such that $\x_0 \in N'$). Hence, since $\dot{\mtcl Q}_{\x_0}$ is a $\mtcl P_{\x_0}$--name for a $V$--finitely proper poset on $\omega_1$, there is an ordinal $\pi^{\ast}$ and an extension $r$ of $\ov q$ forcing that $\pi^{\ast}$ extends $\pi$ in $\dot{\mtcl Q}_{\x_0}$ and is $(\overline{N}[\dot G_{\x_0}],\,\dot{\mtcl Q}_{\x_0})$--generic for all relevant $\overline{N}$ (i.e., such that there some $\beta$ such that $(\overline{N}, \beta)\in \Delta_q$, $\beta\geq \x_0+1$ and $\overline{N}\in \mtcl M_{\x_0+1}$). The reason is that $\ov q$ forces that there is a club $\dot D$ witnessing the $V$--finite properness of $\dot{\mtcl Q}_{\x_0}$, and such that every relevant $\overline{N}$  is in $\dot D$. This $\dot D$ can be taken to be the first club, in the well--order of $H(\k^+)[\dot G_{\x_0}]$ induced by $\lhd$, witnessing  the $V$--finite properness of $\dot{\mtcl Q}_{\x_0}$. (For such a relevant $\overline{N}$ there is some $\overline{N}^\ast \in \mtcl M^\ast_{\x_0+1}$ containing $\lhd$ and such that $\overline{N}=\overline{N}^\ast \cap H(\kappa)$ which, since $\lhd\in \ov{N}^\ast$, implies that $\overline{N}^\ast$ contains a name for $\dot D$. Applying this fact and $(2)_{\x_0}$ we conclude that $\overline{q}$ forces $\overline{N} \in \dot D$.) It follows now from Lemma \ref{amalg1} that there is a $\mtcl P_{\x_0+1}$--condition $q_0$ extending $r$, $q\av_{\x_0+1}$ and $q^\circ\av_{\x_0+1}$.  

For $i$ such that $i+1<r$, we assume inductively that $q_i\in\mtcl P_{\x_i+1}$ extends $q\av_{\x_i+1}$ and $q^\circ\av_{\x_i+1}$, and obtain $q_{i+1}\in\mtcl P_{\x_{i+1}+1}$ from $q_i$ by arguing exactly as in the case $i=0$ with $\x_{i+1}$ instead of $\x_0$ and starting with $q_i$ rather than $q\av_\s$. In the end we obtain $q_{i+1}\in\mtcl P_{\x_{i+1}+1}$ extending both $q\av_{\x_{i+1}+1}$ and $q^\circ\av_{\x_{i+1}+1}$. 

Let $\mu = \xi_{r-1}=  max(supp(q^\circ))$ and let $$q^{\dag} = (F_{q_{r-1}}\cup (F_q \restr [\mu+1,\, \a)), \D_{q_{r-1}} \cup \D_{q^\circ} \cup \D_{q})$$

\begin{claim}
$q^{\dag}$ is a condition in $\mtcl P_{\a}$ extending both $q$ and $q^{\circ}$.
\end{claim}

\begin{proof}
We prove by induction that if $\mu+1 \leq \xi \leq \alpha$, then $q^{\dag}\av_{\xi}$ is in $\mtcl P_{\xi}$ and $q^\dag\av_{\xi} \leq_{\xi} q^\circ|_{\xi}$,  $q\av_{\xi}$. Note that the case $\xi= \mu+1$ follows from the fact that $q_{r-1} \leq_{\mu+1} q^\circ\av_{\mu+1}$,  $q\av_{\mu+1}$ and $q^{\dag}\av_{\mu+1}= q_{r-1}$. Assume now that $\xi$ is the successor of an ordinal $\eta \geq \mu+1$. We show that $q^{\dag}\av_{\eta+1}$ satisfies clause  $(b\,4)$ in the definition of $\mtcl P_{\eta+1}$ (clearly it satisfies the other clauses).  In other words, we must show that if $\eta\in dom(F_q)$, then $q^{\dag}\av_\eta$ forces that $F_q(\eta)$ is $(M[\dot{G}_\eta], \dot{\mtcl Q}_\eta)$--generic for all those $M \in \mtcl M_{\eta+1}$ for which there exists an ordinal $\beta \geq \eta+1$ such that $(M, \beta) \in \D_{q_{r-1}} \cup \D_{q^\circ} \cup \D_{q}$. But such a pair $(M, \beta)$ cannot be in $\D_{q_{r-1}}$, since all markers occurring in side conditions in $q_{r-1} \in \mtcl P_{\mu+1}$ are at most $\mu+1< \eta+1$. On the other hand, (ii) implies that $\eta \in supp(q) \setminus \sigma =supp(q) \setminus (N\cap \alpha)$. So, there is no $M \in dom(\D_{q^\circ})$ such that $M \in \mtcl M_{\eta+1}$ (such a countable $M$ is in $N$, and therefore $M\cap \alpha\sub N\cap \alpha$), and hence $(M, \beta)$ is neither in $\D_{q^\circ}$. We conclude that such a pair $(M, \beta)$ is in $\D_{q}$. By $(b\,4)$ applied to $q\av_{\eta+1}$, we have that $q\av_{\eta}$ (and hence, $q^{\dag}\av_\eta$) forces what we want. Finally note that the inductive hypothesis
$q^\dag\av_{\eta} \leq_{\eta} q^\circ\av_{\eta}$, $q\av_{\eta}$, the definition of $q^\dag$, and the fact that the maximum of the support of $q^{\circ}$ is equal to $\mu< \eta$ together imply that $q^\dag\av_{\eta+1} \leq_{\eta+1} q^\circ \av_{\eta+1}$, $q\av_{\eta+1}$. The case when $\xi$ is limit follows from the inductive hypothesis.

\end{proof}

The above claim finishes the proof of $(2)_\a$ in the present case and the proof of the lemma.
\end{proof}

\begin{corollary}\label{proper}
For every $\a\leq \k$, $\mtcl P_\a$ is proper.
\end{corollary}

Given an ordinal $\a< \k$, we let $\dot G^{+}_{\a}$ be a $\mtcl P_{\a+1}$--name for the collection of all $\nu$ for which there exists a condition $q\in \dot G_{\a+1}$ with $\alpha \in dom(F_q)$ and $F_q(\alpha) = \nu$.

The following lemmas are easy.

\begin{lemma}\label{genn}
If $\a<\k$, then $\mtcl P_{\a+1}$ forces that $\dot G^{+}_\a$ generates a $V[\dot G_\a]$--generic filter over $\dot{\mtcl Q}_\a$.
\end{lemma}

\begin{proof}
It is easy to see that $\dot G^{+}_\a$ is forced to be a set of pairwise compatible $\dot{\mtcl Q}_\a$--conditions, so it suffices to show that $\mtcl P_{\a+1}$ forces $\dot G^{+}_\a\cap D\neq \emptyset$ for every dense subset $D$ of $\dot{\mtcl Q}_\a$ in $V[\dot G_\a]$. For this, note that if $\dot{D}$ is a $\mtcl P_\a$--name for a dense subset of $\dot{\mtcl Q}_\a$, then a consecutive application of Claim \ref{quecosa} and Lemma \ref{amalg1} shows that the set of $q \in \mtcl P_{\a+1}$ with $\alpha \in supp(q)$ and such that  $q\av_\alpha$ forces that $F_q(\alpha)$ is in $\dot{D}$ is a dense subset of  $\mtcl P_{\a+1}$.
\end{proof}

\begin{lemma}\label{kappa} 
$\mtcl P_{\k}$ forces $2^{\al_0}=\k$.
\end{lemma}

\begin{proof}
The inequality $2^{\al_0}\geq \k$ follows for example from the fact that there are $\k$--many ordinals $\alpha < \kappa$ such that $\Phi(\a)$ is Cohen forcing, since Cohen forcing has the c.c.c. 

The inequality $2^{\al_0}\leq \k$ follows from the fact that, by Lemma \ref{cc} together with $\k^{\al_1}=\k$, there are exactly $\k$--many nice $\mtcl P_{\k}$--names for subsets of $\o$ (see for example \cite{KUNEN} for a discussion of nice names and arguments involving counting of nice names).
\end{proof}

We are ready to prove our main theorem.

\textbf{Proof of Theorem \ref{mainthm}.} As we said, our forcing will be $\mtcl P_{\k}$.
By Lemma \ref{cc}, Corollary \ref{proper} and Lemma \ref{kappa}, it suffices to show that $\mtcl P_{\k}$ forces $\textsc{PFA}^{\mbox{fin}}(\o_1)$. But this follows easily from the following claim together with Lemma \ref{genn}.

\begin{claim}
If $\dot{\mtcl Q}$ is a $\mtcl P_{\k}$--name for a nontrivial finitely proper poset defined on $\o_1$ and $(\dot{D}_i)_{i<\o_1}$ is a sequence of $\mtcl P_{\k}$--names for dense subsets of $\dot{\mtcl Q}$, then there is a high enough $\a<\k$ such that $\dot{\mtcl Q}$ and all members of $(\dot{D}_i)_{i<\o_1}$ are $\mtcl P_\a$--names and $\Phi (\alpha)= \dot{\mtcl Q}$ is a $\mtcl P_\a$--name for a $V$--finitely proper forcing with respect to $\dot G_\a$.
\end{claim}

\begin{proof}
By the $\aleph_2$--chain condition of $\mtcl P_\kappa$, together with Corollary \ref{compll1} and with the fact that the relevant information about $\dot{\mtcl Q}$ and $(\dot{D}_i)_{i<\omega_1}$ is decided by a collection of $\aleph_1$--many maximal antichains of $\mtcl P_\kappa$, there is some $\alpha < \kappa$ such that $\dot{\mtcl Q}$ and all $\dot{D}_i$ are  $\mtcl P_\alpha$--names. Furthermore, by the choice of $\Phi$ we may assume $\Phi(\alpha)=\dot{\mtcl Q}$. Therefore we will be done if we show that $\mtcl P_\a$ forces that $\dot{\mtcl Q}$ is $V$--finitely proper with respect to $\dot G_\alpha$. The witnessing club for this can be taken to be any club $D$ consisting of structures of the form $N^\ast\cap H(\k)^V$ where $N^\ast\in\mtcl M^\ast_\kappa$ and $\alpha \in N^\ast$. 

Now let $q\in\mtcl P_\alpha$, let $\{N_i\,:\,i<m\}\sub D$ be a finite set such that $\{(N_i, \a)\,:\,i<m\}\sub \D_q$, and let $\n\in\bigcap_i N_i\cap\omega_1$. Then $$q^\ast:=(F_q, \Delta_q\cup\{(N_i, sup(N_i\cap\k))\,:\,i<m\})$$ is clearly a condition in $\mtcl P_\kappa$ extending $q$ (viewing $q$ as a $\mtcl P_\k$--condition in the natural way). Let $G$ be any generic filter for $\mtcl P_\k$ containing $q^\ast$. For each $i$, since $N_i=N^\ast\cap H(\kappa)^V$ for some $N^\ast \in \mtcl M^\ast_\kappa$ and $q^\ast$ is $\mtcl P_\kappa$--pre-generic for $N_i$, we have that $N_i[G]\cap V= N_i$ by Lemma \ref{horribilis}. By finite properness of $\mtcl Q := \dot{\mtcl Q}_G$ there is then some condition $\nu^\ast<\omega_1$ in $\mtcl Q$ extending $\nu$ and $(N_i[G],\, \mtcl Q)$--generic for all $i$. But since, for all $i < m$, $N_i[G] \cap \omega_1 = N_i \cap \omega_1 = N_i[G\cap\mtcl P_\a] \cap \omega_1$, it follows that $\nu^\ast$ is also $(N_i[G\cap\mtcl P_\alpha],\, \mtcl Q)$--generic for all such $i$. This finishes the proof since then, by Corollary \ref{compll1}, $q$ can be extended to a $\mtcl P_\a$--condition forcing that $\nu^\ast$ is $(N_i[G\cap\mtcl P_\alpha],\, \mtcl Q)$--generic for all $i < m$.
\end{proof}

\section{Applications: $\textsc{PFA}^{\mbox{fin}}_{\o_1}$ and the club filter on $\o_1$}\label{applications}

In this final section we show that $\textsc{PFA}^{\mbox{fin}}_{\o_1}$ implies both $\lnot\textsc{WCG}$ and $\lnot\mho$. It will be convenient to introduce the following natural notion of rank of an ordinal with respect to a set of ordinals.\footnote{This notion of rank will be particularly useful in the proof that $\textsc{PFA}^{\mbox{fin}}_{\o_1}$ implies $\lnot\mho$ (Proposition \ref{fa-vs-mho}).}

\begin{definition}
Given a set $X$ and an ordinal $\d$, we define the \textit{Cantor--Bendixson rank of $\d$ with respect to $X$},
$rank(X, \d)$, by specifying that
\begin{itemize}

\it $rank(X, \d) \geq 1$ if and only if $\d$ is a limit point of ordinals in $X$.

\it If $\m >1$, $rank(X, \d) \geq \m$ if and only and for every $\eta < \mu$, $\d$ is a limit of ordinals $\e$ with $rank(X, \e)\geq\eta$.

\end{itemize}

\end{definition}

A function $F:\o_1\into\o_1$ is normal if it is strictly increasing and continuous. The following two lemmas are easy consequences of our definition of rank.

\begin{lemma}\label{small} Let $A\sub X$ be sets of ordinals and let $\d$ be an ordinal. If $rank(A, \d) < rank(X, \d)$, then $rank(X\setminus A, \d) = rank(X, \d)$.\end{lemma}

\begin{lemma}\label{extend}
Given any strictly increasing finite function $f\sub\o_1\times\o_1$, if $rank(f(\x), f(\x))\geq\x$ for every $\x\in dom(f)$, then $f$ can be extended to a normal function $F:\o_1\into\o_1$.
\end{lemma}

\begin{proof}
It suffices to prove, for all $\x<\o_1$, that if $\x_0<\x$ and $\a<\b<\o_1$ are such that $rank(\a, \a)\geq\x_0$ and $rank(\b, \b)\geq\x$, then there is a strictly increasing and continuous function $h:[\x_0,\,\x]\into [\a,\, \b]$ with $h(\x_0)= \a$ and $h(\x)= \b$. The proof of this fact is immediate by induction on $\x$ and uses the definition of rank.
\end{proof}
 
 Let us first show the following.

\begin{proposition}\label{fa-vs-wcg}
$\textsc{PFA}^{\mbox{fin}}(\o_1)$ implies $\lnot\textsc{WCG}$.
\end{proposition}

\begin{proof}
Let $\mtcl A = \la A_\d\,:\, \d\in Lim(\o_1)\ra$ be a ladder system on $\o_1$. We want to show that there is a club $C\sub\o_1$ such that $C\cap A_\d$ is finite for every limit ordinal $\d\in C$. Let $\mtbb P_{\mtcl A}$ be the following partial order: 

A condition in $\mtbb P_{\mtcl A}$ is a pair $(f, \langle b_{\delta} \,:\, \delta \in D\rangle)$ with the following properties:

\begin{itemize}
\it[(1)] $f\sub\o_1\times\o_1$ is a strictly increasing finite function such that $rank(f(\x), f(\x))\geq\x$ for every $\x\in dom(f)$.

\it[(2)] $D\sub dom(f)\cap Lim(\o_1)$ and for each $\delta \in D$, $b_\d$ is a finite subset of $A_{f(\d)}$ and $range(f) \cap A_{f(\delta)} =b_{\delta}$. 

\end{itemize}

Given $\mtbb P_{\mtcl A}$--conditions $p^\e = (f^{\e}, \langle b^{\e}_{\delta} \,:\, \delta \in D_\e)\rangle)$ for $\e\in\{0, 1\}$, $p_1$ extends $p_0$ if and only if

\begin{itemize}

\it[(i)] $f^0\sub f^1$,

\it[(ii)] $D_0\sub D_1$, and

\it[(iii)] $b^0_\delta = b^1_\delta$ for every  $\delta \in D_0$.

\end{itemize}

The forcing $\mtbb P_{\mtcl A}$ is a natural variation of Baumgartner's forcing for adding a club of $\o_1$ with finite conditions (see \cite{BAU}). It clearly has size $\al_1$, so in order to show that there is a club of $\o_1$ avoiding $\mtcl A$ it will suffice to argue that $\mtbb P_{\mtcl A}$ adds such a club and that it is finitely proper. The proof of the following lemma is completely standard and essentially like in the corresponding proof for Baumgartner's forcing, using Lemma \ref{extend} and the fact that if $(f, \la b_\d\,:\,\d\in D\ra)$ is a condition in $\mtbb P_\mtcl A$ and $\d\in D$, then $rank(f(\d)\setminus A_{f(\d)}, f(\d))\geq\d$ (which is true by Lemma \ref{small}, since $rank(A_{f(\d)}, f(\d)) = 1< \d$).   

\begin{lemma}\label{club}
Let $p=(f, \la b_\d\,:\,\d\in D\ra)\in \mtbb P_{\mtcl A}$. Then the following is true.
\begin{itemize} 
\it[(i)] For all $\beta<\o_1$ there is some $(f', \la b'_\d\,:\,\d\in D'\ra)$ in $\mtbb P_\mtcl A$ extending $p$ and such that $\beta \in dom(f')$. If $\beta$ is a limit ordinal, then we may take $D'$ such that $\beta \in D'$. Furthermore, if $\beta\notin dom(f)$ is such that $f``\b\sub\b$ and $rank(\b, \b)=\b$, then we may take $f'$ such that $f'(\b)=\b$.
\it[(ii)] For every limit ordinal $\a\in dom(f)$ and every $\x< f(\a)$ there is some  $(f', \la b'_\d\,:\,\d\in D'\ra)$ in $\mtbb P_\mtcl A$ extending $p$ and there is some $\b\in dom(f')\cap\a$ such that $f'(\b)>\x$.

\end{itemize} 
\end{lemma}

It follows from Lemma \ref{club} that $\mtbb P_{\mtcl A}$ forces that the union of the ranges of all first components of conditions in the generic filter is a club of $\o_1^V$ and that this club has finite intersection with each $A_\d$.

It remains to show that $\mtbb P_\mtcl A$ is finitely proper. The proof of this is basically the same as the proof that $\mtbb P_\mtcl A$ is proper (which is quite well--known, see \cite{SHELAHPIF}). For completeness we give the proof of finite properness. 

\begin{lemma}\label{wcg-fin-proper} $\mtbb P_{\mtcl A}$ is finitely proper.\end{lemma}

\begin{proof} Let $\{N_i\,:\, i \in m\}$ be a finite set of countable elementary substructures of $H(\o_2)$ such that $\mtbb P_\mtcl A \in N_i$ for all $i$. Since $rank(\d_{N_i}, \d_{N_i}) = \d_{N_i}$ for all $i$, by Fact \ref{club} (i) we know that every condition in $\bigcap_{i<m}N_i$ can be extended to a condition $(f, \la b_\d\,:\,\d\in D\ra)$ such that $\d_{N_i}\in dom(f)$ and $f(\d_{N_i})=\d_{N_i}$ for all $i$. Hence, it will suffice to show that if $p =(f, \langle b_{\delta} \,: \,\delta \in D\rangle)$ is a condition in $\mtbb P_\mtcl A$ and each $\d_{N_i}$ is a fixed point of $f$, then $p$ is $(N_i,\, \mtbb P_\mtcl A)$--generic for all $i$. For this, fix $i< m$, $E$ a dense subset of $\mtbb P_\mtcl A$ in $N_i$, and suppose without loss of generality that $p$ is in $E$. We may also assume that $f\restr\delta_{N_i}$ is nonempty. It suffices to show that $p$ is compatible with a condition in $E\cap N_i$.

For this, let $\m= max(range(f\restr \d_{N_i}))$ and let $g:\o_1\setminus (\m + 1) \into \o_1$ be the function sending each $\n$ in $\o_1\setminus(\m+1)$ to the least $\x$ with the property that there is a condition $p'$ in $E$, $p' =  (f', \langle b'_{\delta} \,: \,\delta \in D' \rangle)$, such that 

\begin{itemize} 

\it[(a)] $p'$ extends $(f\restr\d_{N_i}, \la b_\d\,:\, \d\in D\cap \d_{N_i}\ra)$, 

\it[(b)] $f'\restr\d_{N_i} = f\restr\d_{N_i}$, 

\it[(c)] $\x>\n$, and 

\it[(d)] $\x$ is the least ordinal in the range of $f'$ strictly above $\m$.  

\end{itemize}

Note that for every $\n\in \d_{N_i}\setminus (\m + 1)$, $\d_{N_i}$ and $p$ witness together that the set of pairs $(\x, p')$ satisfying (a)--(d) is nonempty. Hence $g$ is a well--defined function. Note also that $g$, being definable from the condition $(f\restr\d_{N_i}, \la b_\d\,:\, \d\in D\cap \d_{N_i}\ra)$, is in $N_i$ since $f(\d_{N_i}) = \d_{N_i}$ and therefore  $(f\restr\d_{N_i}, \la b_\d\,:\, \d\in D \cap \d_{N_i}\ra)$ is in $N_i$. It follows that the club $C$ of $\eta<\o_1$ such that $g``\eta\sub\eta$ is also in  $N_i$. Now, $C$ has order type $\o_1$, and therefore $C\cap\d_{N_i}$ has order type $\d_{N_i}$ by correctness of $N_i$. Hence, we may find some $\eta\in \d_{N_i}\cap C$ and some $\n<\eta$ such that $[\n,\,\eta]$ has empty intersection with $A_{f(\d)}$ for every $\d\in D$ such that $\d\geq\d_{N_i}$. But then, by definition of $g$ together with the correctness of $N_i$ there is some $p' = (f', \langle b'_{\delta} \,: \,\delta \in D' \rangle)$ in $N_i\cap E$ extending $(f\restr\d_{N_i}, \la b_\d\,:\, \d\in  D \cap \d_{N_i}\ra)$, such that $f'  \restr\d_{N_i} = f\restr\d_{N_i}$, and such that the least ordinal in the range of $f'$ strictly above $\m$ is also above $\n$. But then $f\cup f'$ is a function satisfying condition (1) in the definition of $\mtbb P_\mtcl A$ and, in addition, $range(f\cup f')\cap A_{f(\d)} = range(f)\cap A_{f(\d)} = b_{\d}$ for every $\d\in D$. It then follows that $(f\cup f', \vec b)$, where $\vec b$ is the union of $\la b_\d\,:\,\d\in D\ra$ and $\la b'_\d\,:\,\d\in D'\ra$, is a condition in $\mtbb P_\mtcl A$ extending both $p$ and $p'$.
\end{proof}

Lemma \ref{wcg-fin-proper} completes the proof of the proposition.

\end{proof}

Doing minor modifications to the forcing in the proof of Proposition \ref{fa-vs-wcg} it is easy to derive other similar statements from $\textsc{PFA}^{\mbox{fin}}(\o_1)$. For example one can check that the negation of \emph{Very Weak Club Guessing} ($\textsc{VWCG}$) follows from $\textsc{PFA}^{\mbox{fin}}(\o_1)$, where $\textsc{VWCG}$ is the assertion that there is a collection $\mtcl A$ of size $\al_1$ consisting of subsets of $\o_1$ of order type $\o$ such that every club of $\o_1$ has infinite intersection with
some $A\in\mtcl A$. In other words, $\textsc{VWCG}$ says the same thing as $\textsc{WCG}$ but allowing $\al_1$--many cofinal subsets of $\delta$ for each $\delta \in Lim(\omega_{1})$. One can actually show that $\textsc{PFA}^{\mbox{fin}}(\o_1)$ implies the negation of the even weaker versions of $\textsc{VWCG}$ where one fixes a countable ordinal $\tau$ and considers sets of ordinals of order type at most $\tau$. Specifically, one has the following.

\begin{proposition}\label{strong-wcg-negation} $\textsc{PFA}^{\mbox{fin}}(\o_1)$ implies that for every $\tau<\o_1$ and every set $\mtcl A$, if $\mtcl A$ is a collection of $\aleph_1$--many sets of ordinals of order type at most $\tau$, then there is a club $C\sub\o_1$ having finite intersection with all members of $\mtcl A$.\end{proposition}

The proof of Proposition \ref{strong-wcg-negation}, which we will omit here, is essentially the same as the proof of Proposition \ref{fa-vs-wcg}.

Finally we derive the failure of $\mho$ from $\textsc{PFA}^{\mbox{fin}}(\o_1)$.

\begin{proposition}\label{fa-vs-mho}
$\textsc{PFA}^{\mbox{fin}}(\o_1)$ implies $\lnot\mho$.
\end{proposition}

\begin{proof} 

We will prove that every instance of $\lnot \mho$ follows from $\textsc{PFA}^{\mbox{fin}}(\o_1)$.
Let $\mtcl G=\langle g_\d\,:\,\d\in\o_1 \rangle$ be such that each $g_\d$ is a continuous function
from $\d$ into $\o$ with respect to the order topology. Let $\mtbb P_\mtcl G$ be the forcing notion consisting of all pairs $(f, \langle d_\x\,:\,\x\in D\ra)$ satisfying the following conditions.

\begin{itemize}

\it[(1)] $f\sub\o_1\times\o_1$ is a finite strictly increasing function.

\it[(2)] For every $\x\in dom(f)$, $rank(f(\x), f(\x))\geq\x$.

\it[(3)] $D\sub dom(f)$ and for every $\x\in D$,

\begin{itemize}

\it[(3.1)] $d_\x<\o$,

\it[(3.2)] $g_{f(\x)}``\,range(f)\sub\o\bs\{d_\x\}$, and

\it[(3.3)] $rank(\{\g< f(\x)\,:\, g_{f(\x)}(\g)\neq d_\x\},
f(\x))=rank(f(\x), f(\x))$.

\end{itemize}

\end{itemize}

Given conditions $p_\e=(f_\e, \langle d^\e_\x\,:\,\x\in D_\e \rangle)\in \mtbb P_\mtcl G$ for $\e\in\{0, 1\}$, we say that $p_1$ extends
$p_0$ if and only if $f_0\sub f_1$, $D_0\sub D_1$, and $d^1_\x = d^0_\x$ for all $\x\in D_0$.

Lemma \ref{001} is easy to prove by appealing to condition (2) in the definition of $\mtcl P_\mtcl G$, together with the openness of all $g_{\d}^{-1}(n)$.


\begin{lemma}\label{001}
For every $p=(f, \langle d_\x\,:\,\x\in D \rangle )\in \mtbb P_\mtcl G$ and every $\x_0<\o_1$ there is a condition $p'\in\mtbb P_\mtcl G$
extending $p$ and such that $\x_0\in dom(f^{p'})$. Also, if $\x\in
dom(f)$ is a limit ordinal and $\e<f(\x)$, then there is a condition
$p'\in \mtbb P_\mtcl G$ and some $\x'<\x$ in $dom(f^{p'})$ such that
$f^{p'}(\x')>\e$.
\end{lemma}

\begin{proof}
Let us prove the first claim (the second claim is proved similarly).
We may assume that $\x_0\notin dom(f)$ and that $\x_1=min(D\bs
\x_0)$ exists (otherwise the proof is easier).

Note that for every $\x'>\x_1$ in $D$ there is some $l_{\x'}<\o$,
$l_{\x'}\neq d_{\x'}$, such that $g_{f(\x')}(f(\x_1))=l_{\x'}$.
Since all $g_{f(\x')}^{-1}(\{l_{\x'}\})$ are open in the order
topology, we may fix $\eta<f(\x_1)$ such that
$g_{f(\x')}``[\eta,\,f(\x_1))=\{l_{\x'}\}$ for every
$\x'>\x_1$ in $D$. Let $X=\{\g< f(\x_1)\,:\, g_{f(\x_1)}(\g)\neq
d_{\x_1}\}$.

Since $rank(X, f(\x_1))= rank(X\bs\eta, f(\x_1))=rank(f(\x_1),
f(\x_1))\geq\x_1$, we may find $\g\in[\eta,\,f(\x_1))$ such that
$g_{f(\x_1)}(\g)\neq d_{\x_1}$ and such that $rank(\g, \g)\geq\x_0$.

Now it is easy to check that $p'=(f\cup\{\langle\x_0, \g\rangle\}, \langle d_\x\,:\,\x\in D
\rangle)$ is a condition extending $p$.
\end{proof}

\begin{lemma}\label{promises}
For every $p=(f, (d_\x\,:\,\x\in D))\in \mtbb P_\mtcl G$ and every
$\x\in dom(f)$ there is a condition $p'\in\mtbb P_\mtcl G$ extending
$p$ and such that $\x\in D^{p'}$.
\end{lemma}

\begin{proof}
Fix two distinct colours $d$, $d'$ in $\o\setminus range(g_{f(\x)}\restr range(f))$. If suffices to prove that at least one of 

\begin{itemize}

 \it[(i)] $rank(\{\g < f(\x)\,:\, g_{f(\x)}(\g)\neq d\}, f(\x))=rank(f(\x), f(\x))$ and
 
  \it[(ii)] $rank(\{\g < f(\x)\,:\, g_{f(\x)}(\g)\neq d'\}, f(\x))=rank(f(\x), f(\x))$ \end{itemize}
  
  \noindent holds. But if $rank(\{\g < f(\x)\,:\, g_{f(\x)}(\g)\neq d\}, f(\x)) < rank(f(\x), f(\x))$, then $rank(\{\g < f(\x)\,:\, g_{f(\x)}(\g) = d\}, f(\x)) = rank(f(\x), f(\x))$ by Lemma \ref{small}, and therefore $rank(\{\g < f(\x)\,:\, g_{f(\x)}(\g)\neq d'\}, f(\x)) = rank(f(\x), f(\x))$ since $\{\g<f(\x)\,:\,g_{f(\x)}(\g)=\d\}$ is contained in $\{\g < f(\x)\,:\, g_{f(\x)}(\g)\neq d'\}$.
\end{proof}

It follows from the above lemmas that 
if $G$ is $\mtbb P_\mtcl G$--generic, then the union of the ranges of all first components of conditions in $G$ is a club $C$  of $\o_1^V$ and for every $\d\in C$ there is some $d_\d\in\o$ such that $g_\d``\,C\sub\o\bs\{d_\d\}$.

Obviously, $\mtbb P_\mtcl G$ has cardinality $\al_1$. It remains to show the following. 

\begin{lemma}\label{fin-prop-mho} $\mtbb P_\mtcl G$  is finitely proper.\end{lemma}

\begin{proof} Let $\{N_i\,:\, i \in m\}$ be a finite set of countable elementary substructures of $H(\o_2)$ containing $\mtbb P_\mtcl G$ and let $p =(f, (d_\x\,:\,\x\in D))$ be a condition of $\mtbb P_\mtcl G$ such that for each $i$:

\begin{itemize}

\it[(a)] $\d_{N_i}$ is a fixed point of $f$,

\it[(b)] $\d_{N_i} \in D$, and

\it[(c)] $\{ \beta< \d_{N_i} \,:\, g_{\d_{N_i} }(\beta) \neq d_{\d_{N_i}} \}$ is $N_j$--stationary\footnote{The concept of $M$--stationarity appears in \cite{MOORE}. In our context, saying that $Y\sub\o_1$ is $N$--stationary means that $Y$ intersects all clubs of $\o_1$belonging to $N$.} for every $j \in m$ such that $\d_{N_i}= \d_{N_j}$.
\end{itemize}

\noindent By arguing as in the proof of Lemma \ref{wcg-fin-proper} it is easy to see that such a $p$ is $(N_i, \mtbb P_\mtcl G)$--generic for all $i$. The main point is that if $C\sub\o_1$ is a club in $N_i$ as in the proof of Lemma \ref{wcg-fin-proper}, then $C \cap \{ \beta< \d_{N_i} \,:\, g_{\d_{N_i} }(\beta) \neq d_{\d_{N_i}} \}  \neq \emptyset$. But this is of course ensured by the $N_i$--stationarity of $\{ \beta< \d_{N_i} \,:\, g_{\d_{N_i} }(\beta) \neq d_{\d_{N_i}} \}$.

Since every condition in $\bigcap_{i\leq m}N_i$ can be extended to a condition $p=(f, (d_\x\,:\,\x\in D))$ satisfying (a), the proof of the lemma will be finished once we show that every $p=(f, (d_\x\,:\,\x\in D))$ satisfying (a) can be extended to a $\mtbb P_\mtcl G$--condition $p'$ satisfying also (b) and (c). For this, let $(\d_j)_{j<n}$ be the increasing enumeration of $\{\d_{N_i}\,:\,i<m\}$ and let $(i^j_k)_{j< n,\, k<n_j}$ be such that $\{N_i\,:\,\d_{N_i}=\d_j\}=\{N_{i^j_0},\ldots N_{i^j_{n_j-1}}\}$ for all $j$. For each $j$ let $\{d^j_0,\ldots d^j_{n_j}\}$ be such that $\{d^j_0,\ldots d^j_{n_j}\}\cap g_{\d_j}``\,range(f)=\emptyset$.

\begin{claim} For every $j$ there is some
$d(j) \in \{d^j_0,\ldots d^j_{n_j}\}$ such that $\{\beta< \d_{N_j} \,:\, g_{\d_{N_j} }(\beta) \neq d(j) \}$ is $N_{i^j_{k}}$--stationary for every $k < n_j$.\end{claim}

 \begin{proof} By arguing as in the proof of  Lemma \ref{promises} one can see that for every $k<n_j$ there is some $e_k\in\{d^j_0,\ldots d^j_{n_j}\}$ such that $\{\b < \d_{N_j}\,:\, g_{d_{N_j}}(\b)\neq d\}$ is $N_{i^j_k}$--stationary for every $d\in\{d^j_0,\ldots d^j_{n_j}\} \bs \{e_k\}$. But then, if $d\in\{ d^j_0,\ldots d^j_{n_j}\}\bs\{e_0,\ldots e_{n_j-1}\}$, then $\{\b < \d_{N_j}\,:\, g_{d_{N_j}}(\b)\neq d\}$ is $N_{i^j_k}$--stationary for every $k$.\end{proof}
 
Now we may extend $p$ to a condition $p'$ of the form $(f, (d'_\x\,:\,\x\in D\cup\{\d_0,\ldots \d_{n-1}\}))$ where $d'_\x =  d_\x$ if $\x\in D$ and $d'_{\d_j} = d(j)$ if $j<n$ and $\d_j\notin D$, and $p'$ will satisfy (a)--(c).  The point is that condition (3.3) in the definition of $\mtbb P_\mtcl G$ holds for $p'$ thanks to (c). For this, given any $i$ and any $\n<\d_{N_i}$, let $C\in N_i$ be a club of $\x<\o_1$ such that $rank(\x,\,\x)\geq \n$ and note that $C\cap g_{\d_{N_i}}^{-1}(d_{N_i})\neq\emptyset$. \end{proof}

Lemma \ref{fin-prop-mho} concludes the proof of the proposition.\end{proof}

We do not know whether $\textsc{PFA}^{\mbox{fin}}(\o_1)$ implies $\lnot\mho_n$ for any $n$, $2\leq n<\o$. As a matter of fact, the methods in the present paper do not seem to produce models of $\lnot\mho_n$ for any $n$. The reason is basically that if $N_0,\ldots, N_m$ are countable substructures such that $\d=N_0\cap\o_1=\ldots = N_m\cap\o_1$, $f:\d\into n$, and $n\leq m$, then it need not be true that there is any $i\in m$ such that $f^{-1}(i)$ is $N_j$--stationary for all $j\leq m$. On the other hand, a straightworfard variation of the proof of Proposition \ref{fa-vs-mho} shows that $\lnot\mho_2$ follows from $\textsc{PFA}(\o_1)$.

\end{document}